\documentclass[runningheads]{llncs}
\usepackage{graphicx}
\usepackage{lscape}
\usepackage{pdflscape}

\usepackage{stmaryrd}

\usepackage{cmll}

\usepackage{comment}

\usepackage{amscd}
\usepackage{amssymb,amsmath}
\usepackage{mathptmx}
\usepackage{mathrsfs}
\usepackage{color}
\usepackage{xspace}
\usepackage{bussproofs}
\EnableBpAbbreviations

\usepackage{tikz}

  \usepackage{txfonts}

\usepackage{bpextra}

\usepackage{enumerate}

\usepackage{centernot}

\usepackage{mathtools}

\usepackage{hyperref}

\usepackage{cleveref}

\usepackage{relsize}

\usepackage{caption}

\usepackage{multirow}

\usepackage{multicol}

\usepackage{array}
\newcolumntype{C}[1]{>{\centering\arraybackslash}p{#1}}
\newcolumntype{L}[1]{>{\arraybackslash}p{#1}}

\def\fCenter{{\mbox{$\ \vdash\ $}}}

\newcommand{\scs}{\scriptsize}
\newcommand{\fns}{\footnotesize}

\newcommand{\marginnote}[1]{\marginpar{\raggedright\tiny{#1}}}

\def\mc{\multicolumn}

\newcommand{\mb}{\mathbb}
\newcommand{\mcl}{\mathcal}

\newcommand{\xneg}{\ensuremath{\neg}\xspace}
\newcommand{\xtop}{\ensuremath{\top}\xspace}
\newcommand{\xbot}{\ensuremath{\bot}\xspace}
\newcommand{\xand}{\ensuremath{\wedge}\xspace}
\newcommand{\xor}{\ensuremath{\vee}\xspace}
\newcommand{\xrarr}{\ensuremath{\rightarrow}\xspace}

\newcommand{\XNEG}{\ensuremath{\:\tilde{\neg}}\xspace}
\newcommand{\XTOP}{\hat{\top}}
\newcommand{\XBOT}{\ensuremath{\check{\bot}}\xspace}
\newcommand{\XAND}{\ensuremath{\:\hat{\wedge}\:}\xspace}
\newcommand{\XOR}{\ensuremath{\:\check{\vee}\:}\xspace}

\newcommand{\dneg}{\ensuremath{{\sim}}\xspace}
\newcommand{\rdneg}{\ensuremath{\neg}\xspace}
\newcommand{\dtop}{\ensuremath{1}\xspace}
\newcommand{\dbot}{\ensuremath{0}\xspace}
\newcommand{\dand}{\ensuremath{\cap}\xspace}
\newcommand{\dor}{\ensuremath{\cup}\xspace}

\newcommand{\DNEG}{\ensuremath{\:\tilde{\sim}}\xspace}

\newcommand{\DTOP}{\ensuremath{\hat{1}}\xspace}
\newcommand{\DBOT}{\ensuremath{\check{0}}\xspace}
\newcommand{\DAND}{\ensuremath{\:\hat{\cap}\:}\xspace}
\newcommand{\DOR}{\ensuremath{\:\check{\cup}\:}\xspace}

\newcommand{\XDIANU}{\ensuremath{\langle\hat{\nu}\rangle}\xspace}
\newcommand{\xdianu}{\ensuremath{\langle\nu\rangle}\xspace}
\newcommand{\DBOXUN}{\ensuremath{[\check{\rotatebox[origin=c]{180}{$\nu$}}]}\xspace}
\newcommand{\dboxun}{\ensuremath{[\rotatebox[origin=c]{180}{$\nu$}]}\xspace}

\newcommand{\DDIANI}{\ensuremath{\langle\hat{\ni}\rangle}\xspace}
\newcommand{\ddiani}{\ensuremath{\langle\ni\rangle}\xspace}

\newcommand{\XBOXNIN}{\ensuremath{[\check{\not\in}]}\xspace}
\newcommand{\xboxnin}{\ensuremath{[\not\in]}\xspace}

\newcommand{\XDIAIN}{\ensuremath{\langle\hat{\in}\rangle}\xspace}
\newcommand{\xdiain}{\ensuremath{\langle\in\rangle}\xspace}
\newcommand{\DBOXNI}{\ensuremath{[\check\ni]}\xspace}
\newcommand{\dboxni}{\ensuremath{[\ni]}\xspace}
\newcommand{\DDIANNI}{\ensuremath{\langle\hat{\not\ni}\rangle}\xspace}
\newcommand{\ddianni}{\ensuremath{\langle\not\ni\rangle}\xspace}

\newcommand{\DDIAUNC}{\ensuremath{\langle\hat{\rotatebox[origin=c]{180}{$\nu$}}^c\rangle}\xspace}
\newcommand{\ddiaunc}{\ensuremath{\langle\rotatebox[origin=c]{180}{$\nu$}^c\rangle}\xspace}

\newcommand{\XBOXNUC}{\ensuremath{[\check{\nu^c}]}\xspace}
\newcommand{\xboxnuc}{\ensuremath{[\nu^c]}\xspace}

\newcommand{\MTRA}{\ensuremath{\check{\,\vartriangleright\,}}\xspace}
\newcommand{\mtra}{\ensuremath{\vartriangleright}\xspace}

\newcommand{\MTBRA}{\ensuremath{\check{\,\blacktriangleright\,}}\xspace}
\newcommand{\mtbra}{\ensuremath{\blacktriangleright}\xspace}

\newcommand{\MTAND}{\ensuremath{\,\hat{\blacktriangle}\,}\xspace}
\newcommand{\mtAND}{\ensuremath{\blacktriangle}\xspace}

\newcommand{\DRHDNNI}{\ensuremath{[\check{\not\ni}\rangle}\xspace}
\newcommand{\drhdnni}{\ensuremath{[\not\ni\rangle}\xspace}

\newcommand{\XRHDNIN}{\ensuremath{[\check{\not\in}\rangle}\xspace}
\newcommand{\xrhdnin}{\ensuremath{[\not\in\rangle}\xspace}

\newcommand{\abla}{\nabla}

\usetikzlibrary{arrows}
\usetikzlibrary{matrix}
\usetikzlibrary{patterns}
\usetikzlibrary{shapes}
\usepackage{authblk}

\makeindex

\begin{document}
\title{Non normal logics: semantic analysis and proof theory
}
%
\author{Jinsheng Chen\inst{1} 
\and
Giuseppe Greco\inst{2} 
\and
Alessandra Palmigiano \inst{1,3}\thanks{This research is supported by the NWO Vidi grant 016.138.314, the NWO Aspasia grant 015.008.054, and a Delft Technology Fellowship awarded to the fourth author}
\and
Apostolos Tzimoulis\inst{4}
}
\authorrunning{Chen,  Greco, Palmigiano, Tzimoulis}
%
\institute{
Delft University of Technology, the Netherlands
\and 
University of Utrecht, the Netherlands
\and 
University of Johannesburg, South Africa
\and 
Vrije Universiteit Amsterdam, the Netherlands
}
\maketitle              
\begin{abstract}
We introduce proper display calculi for  basic monotonic modal logic, the conditional logic CK and a number of their axiomatic extensions. These calculi are sound, complete, conservative and enjoy cut elimination and subformula property. Our proposal applies the multi-type methodology in the design of display calculi, starting from a semantic analysis based on the translation from monotonic modal logic to normal bi-modal logic.
\keywords{ Monotonic modal logic \and Conditional logic \and  Proper display calculi.}
\end{abstract}
\section{Introduction}
By {\em non normal logics} we understand in this paper those propositional logics algebraically captured by varieties of {\em Boolean algebra expansions}, i.e.~algebras $\mathbb{A} = (\mathbb{B}, \mathcal{F}^\mathbb{A}, \mathcal{G}^\mathbb{A})$ such that $\mathbb{B}$ is a Boolean algebra, and $\mathcal{F}^\mathbb{A}$ and $\mathcal{G}^\mathbb{A}$ are finite, possibly empty families of operations on $\mathbb{B}$ in which the requirement is dropped that each operation in $\mathcal{F}^\mathbb{A}$ be finitely join-preserving or meet-reversing  in each coordinate and  each operation in $\mathcal{G}^\mathbb{A}$ be finitely meet-preserving or join-reversing  in each coordinate. Very well known examples of non normal logics are {\em monotonic modal logic} \cite{chellas1980modal} and {\em conditional logic} \cite{nute2012topics,chellas1975basic}, which have been intensely investigated, since they capture key aspects of agents' reasoning, such as the epistemic \cite{van2011dynamic}, strategic \cite{pauly2003game,pauly2002modal}, and hypothetical \cite{gabbay2000conditional,lewis2013counterfactuals}.

Non normal logics have been extensively investigated both with model-theoretic tools \cite{hansen2003monotonic} and with proof-theoretic tools \cite{negri2017proof,Olivetti2007ASC}. Specific to proof theory, the main challenge is to endow non normal logics with analytic calculi which can be modularly expanded with additional rules so as to uniformly capture wide classes of axiomatic extensions of the basic frameworks, while preserving key properties such as cut elimination. In this paper, we propose a method to achieve this goal. We will illustrate this method for the two specific signatures of monotonic modal logic and conditional logic. 

Our starting point is the very well known observation that, under the interpretation of the modal connective of monotonic modal logic in neighbourhood frames $\mathbb{F} = (W, \nu)$, the monotonic `box' operation can be understood as the composition of a {\em normal} (i.e.~finitely join-preserving) semantic diamond $\xdianu$ and a {\em normal} (i.e.~finitely meet-preserving) semantic box $\dboxni$. The binary relations $R_\nu$ and $R_\ni$ corresponding to these two normal operators are not defined on one and the same domain, but span over two domains, namely $R_\nu\subseteq W\times \mathcal{P}(W)$ is s.t.~$w R_\nu X$ iff $X\in \nu(w)$ and $R_\ni\subseteq \mathcal{P}(W)\times W$ is s.t.~$X R_\ni w$ iff $w\in X$ (cf.~\cite[Definition 5.7]{hansen2003monotonic}, see also \cite{kracht1999normal,gasquet1996classical}). 
We refine and expand these observations so as to: (a) introduce a semantic environment of two-sorted Kripke frames (cf.~Definition \ref{def:2sorted Kripke frame}) and their heterogeneous algebras  (cf.~Definition \ref{def:heterogeneous algebras}); (b) outline a network of discrete dualities and adjunctions among these semantic structures and the algebras and frames for monotone modal logic and conditional logic (cf.~Propositions \ref{prop:dduality single type},  \ref{prop:dduality multi-type}, \ref{prop:alg characterization of single into multi}, \ref{prop:adjunction-frames}); (c) based on these semantic relationships, introduce  multi-type {\em normal} logics into which the original non normal logics can embed via suitable translations (cf.~Section \ref{sec:embedding}); (d) retrieve well known dual characterization results  for axiomatic extensions of monotone modal logic and conditional logics as instances of general algorithmic correspondence theory for normal  (multi-type)  LE-logics applied to the translated axioms (cf.~Section \ref{sec:ALBA runs}); (e) extract analytic structural rules from the computations of the first order correspondents of the translated axioms, so that, again by general results on  proper display calculi \cite{greco2016unified} applied to multi-type logical frameworks \cite{bilkova2018logic}), the resulting calculi are sound, complete, conservative and enjoy cut elimination and subformula property.

\section{Preliminaries}
\paragraph{Notation.}
\label{ssec:notation}
Throughout the paper, 
the superscript $(\cdot)^c$ denotes the relative complement of the subset of a given set. When the given set is a singleton $\{x\}$, we will write $x^c$ instead of $\{x\}^c$.  
For any binary relation $R\subseteq S\times T$, 
and any $S'\subseteq S$ and $T'\subseteq T$, we  let $R[S']: = \{t\in T\mid (s, t)\in R \mbox{ for some } s\in S'\}$ and $R^{-1}[T']: = \{s\in S\mid (s, t)\in R \mbox{ for some } t\in T'\}$. As usual, we write $R[s]$ and $R^{-1}[t]$ instead of $R[\{s\}]$ and $R^{-1}[\{t\}]$, respectively. 
For any ternary relation $R\subseteq S\times T\times U$ 
and subsets $S'\subseteq S$, $T'\subseteq T$, and $U'\subseteq U$, we also let
{\small{
\begin{itemize}
\item $R^{(0)}[T', U'] =\{s\in S\mid\  \exists t\exists u(R(s, t, u)\ \&\  t\in T' \ \&\ u\in U')\},$
\item $R^{(1)}[S', U'] =\{t\in T\mid\  \exists s\exists u(R(s, t, u)\ \&\  s\in S' \ \&\ u\in U')\},$
\item $R^{(2)}[S', T'] =\{u\in U\mid\  \exists s\exists t(R(s, t, u)\ \&\  s\in S' \ \&\ t\in T')\}.$
\end{itemize}
}}
Any binary relation $R\subseteq S\times T$ gives rise to the
  {\em modal operators} $\langle R\rangle, [R], [R\rangle, \langle R] :\mathcal{P}(T)\to \mathcal{P}(S)$ s.t.~for any $T'\subseteq T$ 
  {\small{
  \begin{itemize}
\item $\langle R\rangle T' : = R^{-1}[T'] = \{s\in S\mid \exists t( s R t \ \&\  t\in T')\}$;
\item  $ [R]T': = (R^{-1}[{T'}^c])^c = \{s\in S\mid \forall t( s R t \ \to\  t\in T')\}$; 
\item  $ [R\rangle T': = (R^{-1}[T'])^c = \{s\in S\mid \forall t( s R t \ \to\  t\notin T')\}$
\item $\langle R] T' : = R^{-1}[{T'}^c] = \{s\in S\mid \exists t( s R t \ \&\  t\notin T')\}$.
\end{itemize}
}}
\noindent By construction, these modal operators are normal.  In particular, $\langle R\rangle$ is completely join-preserving, $[R]$ is completely meet-preserving, $[R\rangle$ is completely join-reversing and $\langle R]$ is completely meet-reversing. Hence, their adjoint maps exist and coincide with $[R^{-1}]\langle R^{-1}\rangle,  [R^{-1}\rangle, \langle R^{-1}]: \mathcal{P}(S)\to \mathcal{P}(T)$, respectively.
Any ternary relation $R\subseteq S\times T\times U$ gives rise to the
  {\em  modal operators} $\mtra_R: \mathcal{P}(T)\times \mathcal{P}(U)\to \mathcal{P}(S)$ and  $\mtAND_R: \mathcal{P}(T)\times \mathcal{P}(S)\to \mathcal{P}(U)$ and $\mtbra_R: \mathcal{P}(S)\times \mathcal{P}(U)\to \mathcal{P}(T)$  s.t.~for any $S'\subseteq S$, $T'\subseteq T$, and $U'\subseteq U$, 
  {\small{
  \begin{itemize}
\item $T' \mtra_R U': =  (R^{(0)}[T', {U'}^c])^c =\{s\in S\mid\  \forall t\forall u(R(s,t,u)\ \&\  t\in T' \Rightarrow u\in U')\}$;
\item $T' \mtAND_R S': =  R^{(2)}[T', S'] =\{u\in U\mid\  \exists t\exists s(R(s,t,u)\ \&\  t\in T' \ \&\  s\in S')\}$;
\item $S' \mtbra_R U': =  (R^{(1)}[S', {U'}^c])^c =\{t\in T\mid\  \forall s\forall u(R(s,t,u)\ \&\  s\in S' \Rightarrow u\in U')\}$.
\end{itemize}
}}
The stipulations above guarantee that these modal operators are normal.  In particular,   $\mtra_R$ and $\mtbra_R$ are completely join-reversing in their first coordinate and completely meet-preserving in their second coordinate, and $\mtAND_R$ is completely join-preserving in both coordinates. These three maps are residual to each other, i.e.~$S'\subseteq T' \mtra_R U'\, $ iff $\, T' \mtAND_R S'\subseteq U'\, $ iff $\, T'\subseteq S' \mtbra_R U'$ for any $S'\subseteq S$, $T'\subseteq T$, and $U'\subseteq U$.

\subsection{Basic monotonic modal logic and conditional logic}
\label{ssec:prelim}

\paragraph{Syntax.} For a countable set of propositional variables $\mathsf{Prop}$, the languages $\mathcal{L}_{\abla}$ and $\mathcal{L}_{>}$ of monotonic modal logic and conditional logic over $\mathsf{Prop}$ are defined as follows:
\[\mathcal{L}_{\abla}\ni \phi ::= p\mid \neg \phi \mid  \phi \land \phi \mid \abla \phi\quad\quad\quad\quad\mathcal{L}_{>}\ni \phi ::= p\mid \neg \phi \mid  \phi \land \phi \mid  \phi> \phi.\]
The connectives $\top, \land,\lor, \to$ and $\leftrightarrow$ are defined as usual. 
The {\em basic monotone modal logic} $\mathbf{L}_{\abla}$ (resp.~{\em basic conditional logic} $\mathbf{L}_{>}$) is a set   of $\mathcal{L}_{\abla}$-formulas  (resp.~$\mathcal{L}_{>}$-formulas) containing the  axioms of classical propositional logic and closed under modus ponens, uniform substitution and M (resp.~RCEA and RCK$_n$ for all $n\geq 0$):
{\footnotesize{
\begin{center}
\AXC{$\varphi \rightarrow \psi$}
\LeftLabel{\tiny{M}}
\UIC{$\abla \varphi \to \abla \psi$}
\DP
\ \ \ 
\AXC{$\varphi \leftrightarrow \psi$}
\LeftLabel{\tiny{RCEA}}
\UIC{$(\varphi > \chi) \leftrightarrow (\psi > \chi)$}
\DP
\ \ \ 
\AXC{$\varphi_1 \wedge{\! \ldots \!}\wedge \varphi_n \rightarrow \psi$}
\LeftLabel{\tiny{RCK$_n$}}
\UIC{$(\chi > \varphi_1) \wedge{\! \ldots \!}\wedge (\chi > \varphi_n) \rightarrow (\chi > \psi)$}
\DP
\end{center}
}}

\paragraph{Algebraic semantics.} A {\em monotone Boolean algebra expansion}, abbreviated as {\em m-algebra} (resp.~{\em conditional algebra}, abbreviated as {\em c-algebra}) is a pair $\mathbb{A} = (\mathbb{B}, \abla^{\mathbb{A}})$ (resp.~$\mathbb{A} = (\mathbb{B}, >^{\mathbb{A}})$)  s.t.~$\mathbb{B}$ is a Boolean algebra and  $\abla^{\mathbb{A}}$ is a unary monotone operation on $\mathbb{B}$ (resp.~$>^{\mathbb{A}}$ is a binary operation on $\mathbb{B}$ which is finitely meet-preserving in its second coordinate). Interpretation of formulas in algebras under assignments $h:\mathcal{L}_{\abla}\to \mathbb{A}$ (resp.~$h:\mathcal{L}_{>}\to \mathbb{A}$)  and validity of formulas in algebras (in symbols: $\mathbb{A}\models\phi$) are defined as usual. By a routine Lindenbaum-Tarski construction one can show that $\mathbf{L}_{\abla}$ (resp.~$\mathbf{L}_{>}$) is sound and complete w.r.t.~the class of m-algebras (resp.~c-algebras). 

\paragraph{Canonical extensions.} The {\em canonical extension} of an m-algebra (resp.~c-algebra)  $\mathbb{A}$ is  $\mathbb{A}^\delta: = (\mathbb{B}^\delta, \abla^{\mathbb{A}^\delta})$ (resp.~$\mathbb{A}^\delta: = (\mathbb{B}^\delta, >^{\mathbb{A}^\delta})$), where $\mathbb{B}^\delta$ is the canonical extension of $\mathbb{B}$ \cite{jonsson1951boolean}, and $\abla^{\mathbb{A}^\delta}$ (resp.~$>^{\mathbb{A}^\delta}$) is the $\pi$-extension of $\abla^{\mathbb{A}}$ (resp.~$>^{\mathbb{A}}$). 
By general results of $\pi$-extensions of maps (cf.~\cite{gehrke2004bounded}), the canonical extension of an m-algebra (resp.~c-algebra) is a {\em perfect} m-algebra (resp.~c-algebra), i.e.~the Boolean algebra $\mathbb{B}$ on which it is based can be identified with a powerset algebra $\mathcal{P}(W)$ up to isomorphism.

\paragraph{Frames and models.}
A {\em neighbourhood frame}, abbreviated as {\em n-frame}  (resp.~{\em conditional frame}, abbreviated as {\em c-frame}) is a pair $\mathbb{F}=(W,\nu)$ (resp.~$\mathbb{F}=(W,f)$) s.t.~$W$ is a non-empty set and $\nu:W\to \mathcal{P}(\mathcal{P}(W))$ is a {\em neighbourhood function} ($f: W\times\mathcal{P}(W)\to \mathcal{P}(W)$ is a {\em selection function}). 
In the remainder of the paper, even if it is not explicitly indicated, we will assume that n-frames  are {\em monotone}, i.e.~s.t.~for every $w\in W$, if $X\in \nu(w)$ and $X\subseteq Y$, then $Y\in \nu(w)$.   For any n-frame  (resp.~c-frame) $\mathbb{F}$, the {\em complex algebra} of $\mathbb{F}$ is $\mathbb{F}^\ast: = (\mathcal{P}(W), \abla^{\mathbb{F}^\ast})$ (resp.~$\mathbb{F}^\ast: = (\mathcal{P}(W),  >^{\mathbb{F}^\ast})$) s.t.~for all $X, Y\in \mathcal{P}(W)$,
\begin{center}
$\abla^{\mathbb{F}^\ast} X: = \{w\mid X\in \nu(w)  \} \quad\quad\quad \quad X >^{\mathbb{F}^\ast} Y: =  \{w\mid  f(w, X)\subseteq Y\}. $\end{center}
The complex algebra of an n-frame (resp.~c-frame) is an m-algebra (resp.~a c-algebra).
{\em Models} are pairs $\mathbb{M} = (\mathbb{F}, V)$ such that $\mathbb{F}$ is a frame and  $V:\mathcal{L} \to \mathbb{F}^\ast$ is a homomorphism of the appropriate type. Hence, truth of formulas at states in models is defined as $\mb{M},w \Vdash \varphi$ iff $w\in V(\varphi)$, and unravelling this stipulation   for $\abla$- and $>$-formulas, we get:
\[
\mb{M},w \Vdash \abla \varphi \quad \text{iff}\quad  V(\varphi)\in \nu(w) \quad\quad\quad\quad \mb{M},w \Vdash \varphi> \psi \quad \text{iff}\quad  f(w, V(\varphi))\subseteq V(\psi).
\]
Global satisfaction (notation: $\mathbb{M}\Vdash\phi$) and frame validity (notation: $\mathbb{F}\Vdash\phi$) are defined in the usual way. Thus, by definition, $\mathbb{F}\Vdash\phi$ iff $\mathbb{F}^\ast\models \phi$, from which the soundness of $\mathbf{L}_{\abla}$ (resp.~$\mathbf{L}_{>}$) w.r.t.~the corresponding class of frames  immediately follows from the algebraic soundness. Completeness  follows from algebraic completeness, by observing that (a) the canonical extension of any algebra refuting $\phi$ will also refute $\phi$;  (b) canonical extensions are perfect algebras; (c) perfect algebras can be associated with frames as follows: for any $\mathbb{A} = (\mathcal{P}(W), \abla^{\mathbb{A}})$ (resp.~$\mathbb{A} = (\mathcal{P}(W), >^{\mathbb{A}})$) let $\mathbb{A}_\ast:=(W,\nu_{\abla})$ (resp.~$\mathbb{A}_\ast:=(W,f_{>})$) s.t.~for all $w\in W$ and $X\subseteq W$, 
\[\nu_{\abla}(w): = \{X\subseteq W\mid w\in \abla X\}\quad\quad\quad\quad f_{>}(w, X): = \bigcap\{Y\subseteq W\mid w\in X> Y\}.
\]
If $X\in \nu_{\abla}(w)$ and $X\subseteq Y$, then the monotonicity of $\abla$ implies that $\abla X\subseteq \abla Y$ and hence $Y\in \nu_{\abla}(w)$, as required. By construction, $\mathbb{A}\models\phi$ iff $\mathbb{A}_\ast\Vdash \phi$. This is enough to derive the frame completeness of $\mathbf{L}_{\abla}$ (resp.~$\mathbf{L}_{>}$) from its algebraic completeness.

\begin{proposition} 
\label{prop:dduality single type}If $\mathbb{A}$ is a perfect m-algebra (resp.~c-algebra)   and $\mathbb{F}$ is an  n-frame  (resp.~c-frame), then
 $(\mathbb{F}^\ast)_\ast\cong \mathbb{F}$ and  $(\mathbb{A}_\ast)^\ast\cong \mathbb{A}$.
\end{proposition}

\paragraph{Axiomatic extensions.} A {\em monotone modal logic} (resp.~a {\em conditional logic}) is any  extension of $\mathbf{L}_{\abla}$ (resp.~$\mathbf{L}_{>}$)   with $\mathcal{L}_{\abla}$-axioms (resp.~$\mathcal{L}_{>}$-axioms). Below we collect correspondence results  for axioms that have cropped up in the literature \cite[Theorem 5.1]{hansen2003monotonic} \cite{Olivetti2007ASC}. 
	
\begin{theorem}
\label{theor:correspondence-noAlba}
For every n-frame (resp.~c-frame) $\mathbb{F}$,

{\hspace{-0.7cm}
{\small{
\begin{tabular}{@{}rl c l}
N\; & $\mathbb{F}\Vdash\abla \top\quad$ & iff &   $\quad \mathbb{F}\models\forall w[W \in \nu (w)]$\\
P\; & $\mathbb{F}\Vdash\neg \abla \bot\quad$ & iff &  $\quad \mathbb{F}\models\forall w[\varnothing \not \in \nu(w)]$\\
C\; & $\mathbb{F}\Vdash \abla p \land \abla q \to \abla(p \land q)\quad$ & iff &  $\quad \mathbb{F}\models\forall w \forall X \forall Y  [(X \in \nu(w)\ \&\ Y \in \nu (w) )\Rightarrow X \cap Y \in \nu(w)]$\\
T\; & $\mathbb{F}\Vdash\abla p \to p \quad$ & iff &   $\quad \mathbb{F}\models\forall w  \forall X[X \in \nu(w) \Rightarrow w \in X]$\\
4\; & $\mathbb{F}\Vdash\abla \abla  p \to \abla p\quad$ & iff &   $\quad \mathbb{F}\models \forall w   \forall Y \forall X[(X \in \nu (w)\ \&\ \forall x( x\in X\Rightarrow Y \in \nu(x))) \Rightarrow Y \in \nu(w)]$\\
4'\; & $\mathbb{F}\Vdash\abla p \to \abla \abla p\quad$ & iff &  $\quad \mathbb{F}\models\forall w \forall X[X \in \nu(w) \Rightarrow \{y \mid X \in \nu(y)\} \in \nu(w)]$\\
5\; & $\mathbb{F}\Vdash\neg \abla \neg p \to \abla \neg \abla \neg p \quad$ & iff &  $\quad \mathbb{F}\models \forall w \forall X [X \notin \nu(w) \Rightarrow  \{y \mid X \in \nu(y)\}^c \in \nu(w)]$\\
B\; & $\mathbb{F}\Vdash p \to \abla \neg \abla \neg p \quad$ & iff &   $\quad \mathbb{F}\models \forall w\forall X[w \in X \Rightarrow \{y \mid  X^c \in \nu (y)\}^c \in \nu(w)]$\\
D\; & $\mathbb{F}\Vdash \abla p \to \neg \abla \neg p\quad$ & iff &   $\quad \mathbb{F}\models \forall w \forall X[X\in \nu(w)\Rightarrow  X^c \not \in \nu(w)]$\\
CS\; &  $\mathbb{F}\Vdash (p\wedge q) \to (p > q)\quad$ & iff &  $\quad \mathbb{F}\models\forall x\forall Z[f(x,Z)\subseteq \{x\}]$\\
CEM\;&  $\mathbb{F}\Vdash (p > q) \vee (p> \neg q)\quad$ & iff &  $\quad \mathbb{F}\models\forall X \forall y[|f(y,X)|\leq 1]$\\

ID\; & $\mathbb{F}\Vdash p >  p \quad$ & iff &  $\quad\mathbb{F}\models\forall x\forall Z[f(x,Z)\subseteq Z].$\\
\end{tabular}
}}
}
\end{theorem}

\section{Semantic analysis}

\subsection{Two-sorted Kripke frames and their discrete duality}
\label{ssec:2sorted Kripke frame}
Structures similar to those below are considered implicitly in \cite{hansen2003monotonic}, and explicitly in \cite{frittella2017dual}.
\begin{definition}\label{def:2sorted Kripke frame}
A {\em two-sorted n-frame} (resp.~{\em c-frame}) is a structure $\mathbb{K}: = (X, Y, R_\ni, R_{\not\ni}, R_\nu, R_{\nu^c})$ (resp.~$\mathbb{K}: = (X, Y, R_\ni, R_{\not\ni}, T_f)$) such that $X$ and $Y$ are nonempty sets, $R_\ni, R_{\not\ni}\subseteq Y\times X$ and $R_\nu, R_{\nu^c}\subseteq X\times Y$ and $T_f\subseteq X\times Y\times X$. 
Such an n-frame is {\em supported} if for every $D\subseteq X$,
\begin{equation}
\label{eq:supported n-frames}
R_{\nu}^{-1}[(R_{\ni}^{-1}[D^c])^c] = (R_{\nu^c}^{-1}[(R_{\not\ni}^{-1}[D])^c])^c.
\end{equation}
For any two-sorted n-frame  (resp.~c-frame) $\mathbb{K}$, the {\em complex algebra} of $\mathbb{K}$ is \[\mathbb{K}^+: = (\mathcal{P}(X), \mathcal{P}(Y), \dboxni^{\mathbb{K}^+}, \ddianni^{\mathbb{K}^+}, \xdianu^{\mathbb{K}^+}, \xboxnuc^{\mathbb{K}^+}) \quad  \text{(resp.~}\mathbb{K}^+: = (\mathcal{P}(X), \mathcal{P}(Y), \dboxni^{\mathbb{K}^+}, \drhdnni^{\mathbb{K}^+}, \mtra^{\mathbb{K}^+})\text{), s.t.}\] 
{\small{
\begin{center}
\begin{tabular}{r  cr c r}
$\xdianu^{\mathbb{K}^+}: \mathcal{P}(Y)\to \mathcal{P}(X)$ &$\quad$ & $\dboxni^{\mathbb{K}^+}: \mathcal{P}(X)\to \mathcal{P}(Y)$ &$\quad$ &$\ddianni^{\mathbb{K}^+}: \mathcal{P}(X)\to \mathcal{P}(Y)$\\ 
$U\mapsto R^{-1}_\nu[U]$ && $D\mapsto (R_\ni^{-1}[D^c])^c$ && $D\mapsto R_{\not\ni}^{-1}[D]$ \\
&&&\\
$\xboxnuc^{\mathbb{K}^+}: \mathcal{P}(Y)\to \mathcal{P}(X)$ &&$\drhdnni^{\mathbb{K}^+}: \mathcal{P}(X)\to \mathcal{P}(Y)$&$\quad$ & $\mtra^{\mathbb{K}^+}: \mathcal{P}(Y)\times \mathcal{P}(X)\to \mathcal{P}(X)$ \\
 $U\mapsto (R^{-1}_{\nu^c}[U^c])^c$ &&$D\mapsto (R_{\not\ni}^{-1}[D])^c$&& $(U, D)\mapsto (T_f^{(0)}[U, D^c])^c$\\
\end{tabular}
\end{center}
}}
\end{definition}
The adjoints and residuals of the maps above (cf.~Section \ref{ssec:notation}) are defined as follows:
{\small{
\begin{center}
\begin{tabular}{r  cr c rcr}
$\dboxun^{\mathbb{K}^+}: \mathcal{P}(X)\to \mathcal{P}(Y)$ &$\quad$ & $\xdiain^{\mathbb{K}^+}: \mathcal{P}(Y)\to \mathcal{P}(X)$ &$\quad$ &$\xboxnin^{\mathbb{K}^+}: \mathcal{P}(Y)\to \mathcal{P}(X)$ \\
$D\mapsto (R_\nu[D^c])^c$ && $U\mapsto R_\ni[U]$ &&$U\mapsto (R_{\not\ni}[U^c])^c$ \\
&&\\

$\ddiaunc^{\mathbb{K}^+}: \mathcal{P}(X)\to \mathcal{P}(Y)$ && $\xrhdnin^{\mathbb{K}^+}: \mathcal{P}(Y)\to \mathcal{P}(X)$ &$\quad$ & $\mtbra^{\mathbb{K}^+}: \mathcal{P}(X)\times \mathcal{P}(X)\to \mathcal{P}(Y)$ \\ 

 $D\mapsto R_{\nu^c}[D]$   && $U\mapsto (R_{\not\ni}[U])^c$&& $(C, D)\mapsto (T_f^{(1)}[C, D^c])^c$\\

&& $\mtAND^{\mathbb{K}^+}: \mathcal{P}(Y)\times \mathcal{P}(X)\to \mathcal{P}(X)$ \\ 
&& $(U, D)\mapsto T_f^{(2)}[U, D]$  \\
\end{tabular}
\end{center}
}}
Complex algebras of two-sorted frames can be recognized as heterogeneous algebras (cf.~\cite{birkhoff1970heterogeneous}) of the following kind:
\begin{definition}
\label{def:heterogeneous algebras}
A {\em heterogeneous m-algebra} (resp.~{\em c-algebra}) is a structure \[\mathbb{H}: = (\mathbb{A}, \mathbb{B}, \dboxni^{\mathbb{H}},  \ddianni^{\mathbb{H}}, \xdianu^{\mathbb{H}}, \xboxnuc^{\mathbb{H}}) \quad\quad\text{(resp.~}\mathbb{H}: = (\mathbb{A}, \mathbb{B}, \dboxni^{\mathbb{H}}, \drhdnni^{\mathbb{H}}, \mtra^{\mathbb{H}})\text{)}\]
such that $\mathbb{A}$ and $\mathbb{B}$ are Boolean algebras, $\xdianu^{\mathbb{H}}, \xboxnuc: \mathbb{B}\to \mathbb{A}$  are finitely join-preserving and finitely meet-preserving respectively,   $\dboxni^{\mathbb{H}}, \drhdnni^{\mathbb{H}}, \ddianni^{\mathbb{H}}: \mathbb{A}\to \mathbb{B}$ are finitely meet-preserving, finitely join-reversing, and finitely join-preserving respectively, and $\mtra^{\mathbb{H}}:\mathbb{B}\times \mathbb{A}\to \mathbb{A}$ is finitely join-reversing in its first coordinate and finitely meet-preserving in its second coordinate. 
Such an $\mathbb{H}$ is {\em complete} if $\mathbb{A}$ and $\mathbb{B}$ are complete Boolean algebras and the operations above enjoy the complete versions of the finite preservation properties indicated above, and is {\em perfect} if it is complete and $\mathbb{A}$ and $\mathbb{B}$ are perfect.  
The {\em canonical extension} of a heterogeneous m-algebra (resp.~c-algebra)  $\mathbb{H}$ is \mbox{$\mathbb{H}^\delta: = (\mathbb{A}^\delta, \mathbb{B}^\delta, \dboxni^{\mathbb{H}^\delta}, \ddianni^{\mathbb{H}^\delta}, \xdianu^{\mathbb{H}^\delta}, \xboxnuc^{\mathbb{H}^\delta})$} (resp.~$\mathbb{H}^\delta: = (\mathbb{A}^\delta, \mathbb{B}^\delta, \dboxni^{\mathbb{H}^\delta}, \drhdnni^{\mathbb{H}^\delta}, \mtra^{\mathbb{H}^\delta})$), where $\mathbb{A}^\delta$ and $\mathbb{B}^\delta$ are the canonical extensions of $\mathbb{A}$ and $\mathbb{B}$ respectively \cite{jonsson1951boolean}, moreover $\dboxni^{\mathbb{H}^\delta}$, $ \drhdnni^{\mathbb{H}^\delta}$, $\xboxnuc^{\mathbb{H}^\delta}, \mtra^{\mathbb{H}^\delta}$  are the $\pi$-extensions of $\dboxni^{\mathbb{H}}, \drhdnni^{\mathbb{H}}, \xboxnuc^{\mathbb{H}}, \mtra^{\mathbb{H}}$  respectively, and $\xdianu^{\mathbb{H}^\delta},$ $\ddianni^{\mathbb{H}^\delta}$ are the $\sigma$-extensions of $\xdianu^{\mathbb{H}},\ddianni^{\mathbb{H}}$ respectively.
\end{definition}
\begin{definition}
A heterogeneous m-algebra $\mathbb{H}: = (\mathbb{A}, \mathbb{B}, \dboxni^{\mathbb{H}},  \ddianni^{\mathbb{H}}, \xdianu^{\mathbb{H}}, \xboxnuc^{\mathbb{H}})$ is {\em supported}
if $\xdianu^{\mathbb{H}} \dboxni^{\mathbb{H}}a = \xboxnuc^{\mathbb{H}}\ddianni^{\mathbb{H}} a$ for every $a\in \mathbb{A}$.
\end{definition}
It immediately follows from the definitions that 
\begin{lemma}
The complex algebra of a supported two-sorted n-frame is a heterogeneous  supported m-algebra. 
\end{lemma}
\begin{definition}
If  $\mathbb{H} = (\mathcal{P}(X), \mathcal{P}(Y), \dboxni^{\mathbb{H}},  \ddianni^{\mathbb{H}}, \xdianu^{\mathbb{H}}, \xboxnuc^{\mathbb{H}})$ is a  perfect heterogeneous m-algebra
 (resp.~$\mathbb{H} = (\mathcal{P}(X), \mathcal{P}(Y), \dboxni^{\mathbb{H}}, \drhdnni^{\mathbb{H}}, \mtra^{\mathbb{H}})$ is a perfect heterogeneous~c-algebra), its associated two-sorted n-frame (resp.~c-frame) is
 \[\mathbb{H}_+: = (X, Y, R_\ni, R_{\not\ni}, R_\nu, R_{\nu^c})\quad\quad \text{(resp.~}\mathbb{H}_+: = (X, Y, R_\ni, R_{\not\ni}, T_f) \text{), s.t.}\]
 {\small{
\begin{itemize}
\item  $R_{\ni}\subseteq Y\times X$ is defined by $yR_\ni x$ iff $y\notin \dboxni^{\mathbb{H}}x^c$,
\item  $R_{\not\ni}\subseteq Y\times X$ is defined by $xR_{\not\ni} y$ iff $y\in \ddianni^{\mathbb{H}}\{x\}$  (resp.~$y\notin \drhdnni^{\mathbb{H}}\{x\}$),
\item $R_\nu\subseteq X\times Y$ is defined by $xR_\nu y$ iff $x\in \xdianu^{\mathbb{H}}\{y\}$, 
\item $R_{\nu^c}\subseteq X\times Y$ is defined by $xR_{\nu^c} y$ iff $x\notin \xboxnuc^{\mathbb{H}}y^c$, 
\item $T_f\subseteq X\times Y\times X$ is defined by $(x', y, x)\in T_f$ iff $x'\notin  \{y\}\mtra^{\mathbb{H}} x^c$. 
\end{itemize} 
}}
\end{definition}
From the definition above it readily follows that:
\begin{lemma}
If  $\mathbb{H}$ is a  perfect supported  heterogeneous m-algebra, then $\mathbb{H}_+$ is a supported two-sorted n-frame.
\end{lemma}
The theory of canonical extensions (of maps) and the  duality between perfect BAOs and Kripke frames can be readily extended to the present two-sorted case. The following proposition collects these well known facts, the proofs of which are analogous to those of the single-sort case, hence  are omitted.  
\begin{proposition} \label{prop:dduality multi-type}
For every heterogeneous m-algebra (resp.~c-algebra)  $\mathbb{H}$ and every two-sorted n-frame  (resp.~c-frame) $\mathbb{K}$,
\begin{enumerate}
\item $\mathbb{H}^\delta$ is a perfect heterogeneous m-algebra (resp.~c-algebra);
\item $\mathbb{K}^+$ is a perfect heterogeneous m-algebra (resp.~c-algebra);
\item $(\mathbb{K}^+)_+\cong \mathbb{K}$, and if $\mathbb{H}$ is  perfect, then $(\mathbb{H}_+)^+\cong \mathbb{H}$.
\end{enumerate}
\end{proposition}

\subsection{Equivalent representation of  m-algebras  and c-algebras}

Every supported heterogeneous m-algebra (resp.~c-algebra) can be associated with an m-algebra (resp.~a c-algebra) as follows:
\begin{definition}
For every  supported heterogeneous m-algebra  $\mathbb{H} = (\mathbb{A},\mathbb{B}, \dboxni^{\mathbb{H}}, \ddianni^{\mathbb{H}}, \xdianu^{\mathbb{H}}, \xboxnuc^{\mathbb{H}})$ 
 (resp.~c-algebra $\mathbb{H} = (\mathbb{A},\mathbb{B}, \dboxni^{\mathbb{H}}, \drhdnni^{\mathbb{H}}, \mtra^{\mathbb{H}})$),  let $\mathbb{H}_\bullet:  = (\mathbb{A}, \abla^{\mathbb{H}_\bullet})$ (resp.~$\mathbb{H}_\bullet:  = (\mathbb{A}, >^{\mathbb{H}_\bullet})$), where for every $a\in\mathbb{A}$ (resp.~$a, b\in \mathbb{A}$),  
 \[\abla^{\mathbb{H}_\bullet} a = \xdianu^{\mathbb{H}}\dboxni^{\mathbb{H}} a = \xboxnuc^{\mathbb{H}}\ddianni^{\mathbb{H}} a\quad\quad \text{ (resp.~}a >^{\mathbb{H}_\bullet}b: = (\dboxni^{\mathbb{H}} a \wedge \drhdnni^{\mathbb{H}} a)\mtra^{\mathbb{H}} b\text{)}.\]
%
%
\end{definition}
It immediately follows from the stipulations above that $\abla^{\mathbb{H}_\bullet}$ is a monotone map (resp.~$>^{\mathbb{H}_\bullet}$ is finitely  meet-preserving in its second coordinate), and hence $\mathbb{H}_\bullet$ is an m-algebra (resp.~a c-algebra).
Conversely, every complete m-algebra (resp.~c-algebra) can be associated with a supported heterogeneous m-algebra (resp.~a c-algebra) as follows:
\begin{definition}
For every  complete m-algebra  $\mathbb{C} = (\mathbb{A},\abla^{\mathbb{C}})$ (resp.~complete c-algebra  $\mathbb{C} = (\mathbb{A},>^{\mathbb{C}})$),  let $\mathbb{C}^\bullet:  = (\mathbb{A}, \mathcal{P}(\mathbb{A}), \dboxni^{\mathbb{C}^\bullet}, \ddianni^{\mathbb{C}^\bullet}, \xdianu^{\mathbb{C}^\bullet}, \xboxnuc^{\mathbb{C}^\bullet})$ (resp.~$\mathbb{C}^\bullet:  = (\mathbb{A}, \mathcal{P}(\mathbb{A}), \dboxni^{\mathbb{C}^\bullet},$ \mbox{$ \drhdnni^{\mathbb{C}^\bullet},$} $\mtra^{\mathbb{C}^\bullet})$), where for every $a\in\mathbb{A}$ and $B\in \mathcal{P}(\mathbb{A})$,
{\small{
\[\dboxni^{\mathbb{C}^\bullet} a: = \{b\in\mathbb{A}\mid b\leq a\}\quad \quad\xdianu^{\mathbb{C}^\bullet}B: = \bigvee \{\abla^{\mathbb{C}} b\mid b\in B\} \quad\quad  \drhdnni^{\mathbb{C}^\bullet} a: = \{b\in\mathbb{A}\mid a\leq b\}\]
\[\xboxnuc^{\mathbb{C}^\bullet} B: =\bigwedge \{\abla^{\mathbb{C}}b\mid b\notin B\} \quad B \mtra^{\mathbb{C}^\bullet} a: = \bigwedge\{b >^{\mathbb{C}} a\mid b\in B\}\quad  \ddianni^{\mathbb{C}^\bullet} a: = \{b\in\mathbb{A}\mid a\nleq b\}.\]
}}
\end{definition}
%
One can readily see that the operations defined above are all normal by construction, and that they enjoy the complete versions of the preservation properties indicated in Definition \ref{def:heterogeneous algebras}.  Moreover, $\xdianu^{\mathbb{C}^\bullet}\dboxni^{\mathbb{C}^\bullet} a = \abla^{\mathbb{C}} a = \xboxnuc^{\mathbb{C}^\bullet}\ddianni^{\mathbb{C}^\bullet} a$ for every $a\in \mathbb{A}$. Hence,
\begin{lemma}
If $\mathbb{C}$  is a  complete m-algebra  (resp.~complete c-algebra), then
  $\mathbb{C}^\bullet$ is a complete supported heterogeneous m-algebra (resp.~c-algebra). 
\end{lemma}
The assignments $(\cdot)^\bullet$ and $(\cdot)_\bullet$ can be extended to functors between the appropriate  categories of single-type and heterogeneous algebras and their homomorphisms. These functors are adjoint to each other and form a section-retraction pair. Hence:
\begin{proposition}
\label{prop:alg characterization of single into multi}
If   $\mathbb{C}$ is a complete  m-algebra (resp.~c-algebra), then   $\mathbb{C} \cong (\mathbb{C}^\bullet)_\bullet$. Moreover, if $\mathbb{H}$ is a complete supported heterogeneous m-algebra (resp.~c-algebra), then $\mathbb{H}\cong \mathbb{C}^\bullet$ for some complete  m-algebra (resp.~c-algebra) $\mathbb{C}$ iff $\mathbb{H} \cong (\mathbb{H}_\bullet)^\bullet$.
\end{proposition}
The proposition above characterizes up to isomorphism the supported heterogeneous  m-algebras (resp.~c-algebras) which arise from single-type m-algebras (resp.~c-algebras). 
Thanks to the discrete dualities discussed in Sections \ref{ssec:prelim} and \ref{ssec:2sorted Kripke frame}, we can transfer this algebraic characterization to the side of frames, as detailed in the next subsection.

\subsection{Representing n-frames and c-frames as two-sorted Kripke frames}
\begin{definition}
For any n-frame (resp.~c-frame) $\mathbb{F}$, we let $\mathbb{F}^\star: = ((\mathbb{F}^\ast)^\bullet)_+$, and for every supported two-sorted n-frame (resp.~c-frame) $\mathbb{K}$, we let $\mathbb{K}_\star: = ((\mathbb{K}^+)_\bullet)_\ast$.
\end{definition}
Spelling out the definition above, if $\mathbb{F}=(W,\nu)$ (resp.~$\mathbb{F}=(W, f)$) then  $\mathbb{F}^\star = (W,\mathcal{P} (W), R_\ni, R_{\not\ni}, R_{\nu}, R_{\nu^c})$  (resp.~$\mathbb{F}^\star = (W,\mathcal{P} (W),R_{\not\ni}, R_{\ni}, T_f)$) where:
{\small{
\begin{itemize}
\item $R_{\nu} \subseteq W\times \mathcal{P}(W)$ is defined as $x R_{\nu}  Z$ iff $Y\in \nu(x)$;
\item $R_{\nu^c} \subseteq W\times \mathcal{P}(W)$ is defined as $x R_{\nu^c}  Z$ iff $Z\notin \nu(x)$;
\item  $R_{\ni} \subseteq \mathcal{P}(W) \times W$ is defined as $Z R_{\ni}  x$ iff $x\in Z$;
\item  $R_{\not\ni} \subseteq \mathcal{P}(W) \times W$ is defined as $Z R_{\not\ni}  x$ iff $x\notin Z$;
\item $T_f\subseteq W\times \mathcal{P}(W) \times W$ is defined as $T_f(x, Z, x')$ iff $x'\in f(x, Z)$.
\end{itemize}
}}
Moreover, if $\mathbb{K} = (X, Y, R_\ni, R_{\not\ni}, R_{\nu}, R_{\nu^c})$ (resp.~$\mathbb{K} = (X, Y, R_\ni, R_{\not\ni}, T_f)$), then $\mathbb{K}_{\star} = (X, \nu_{\star})$
(resp.~$\mathbb{K}_\star = (X, f_\star)$) where:
{\small{
\begin{itemize}
\item $\nu_\star (x)  = \{D\subseteq X\mid x\in R_\nu^{-1}[(R_\ni^{-1}[D^c])^c] \} = \{D\subseteq X\mid x\in (R_{\nu^c}^{-1}[(R_{\not\ni}^{-1}[D])^c])^c\}$;
\item $f_\star(x, D)  
=  \bigcap \{C\subseteq X\mid x\in T^{(0)}_f[ \{C\}, D^c]\}$.
\end{itemize}
}}
\begin{lemma}
If $\mathbb{F} = (W,\nu)$ is an n-frame, then  $\mathbb{F}^\star$ is a supported two-sorted n-frame. 
\end{lemma}
\begin{proof}
 By definition, $\mathbb{F}^\star$ is a two-sorted n-frame. Moreover, for any $D\subseteq W$,
{\fns{
\begin{center}
\begin{tabular}{clll}
$(R_{\nu^c}^{-1}[(R_{\not\ni}^{-1}[D])^c])^c$ &  = & $ \{w\mid \forall X(X\notin \nu(w)\Rightarrow \exists u(X\not\ni u\ \&\ u\in D))\}$\\
&  = & $ \{w\mid \forall X(X\notin \nu(w)\Rightarrow   D \not\subseteq  X)\}$\\
&  = & $ \{w\mid \forall X(D\subseteq  X \Rightarrow X\in \nu(w) )\}$\\
&  = & $ \{w\mid \exists X(X\in \nu(w) \ \&\  X\subseteq  D)\}$ & ($\ast$)\\
&  = & $R_{\nu}^{-1}[(R_{\ni}^{-1}[D^c])^c].$\\
\end{tabular}
\end{center}
}}
To show the identity marked with $(\ast)$, from top to bottom, take $X:= D$; conversely, if $D\subseteq Z$ then $X\subseteq Z$, and since  by assumption $X\in \nu(w)$ and $\nu(w)$ is upward closed, we conclude that $Z\in \nu(w)$, as required.
\end{proof}

The next proposition is the frame-theoretic counterpart of Proposition \ref{prop:alg characterization of single into multi}. 
\begin{proposition}
\label{prop:adjunction-frames}
If   $\mathbb{F}$ is an n-frame (resp.~c-frame), then   $\mathbb{F} \cong (\mathbb{F}^\star)_\star$. Moreover, if $\mathbb{K}$ is a supported two-sorted n-frame (resp.~c-frame), then $\mathbb{K}\cong \mathbb{F}^\star$ for some n-frame (resp.~c-frame) $\mathbb{F}$ iff $\mathbb{K} \cong (\mathbb{K}_\star)^\star$.
\end{proposition}
\section{Embedding non-normal logics into two-sorted normal logics}
\label{sec:embedding}
The two-sorted frames and heterogeneous algebras discussed in the previous section serve as semantic environment for the multi-type  languages defined below.

\paragraph{Multi-type languages.} For a denumerable set $\mathsf{Prop}$ of atomic propositions, the languages $\mathcal{L}_{MT\abla}$ and $\mathcal{L}_{MT>}$ in  types $\mathsf{S}$ (sets) and $\mathsf{N}$ (neighbourhoods) over $\mathsf{Prop}$ are defined as follows:
{\small
\begin{center}
$\begin{array}{lll}
\mathsf{S} \ni A::= p  \mid \top \mid \bot \mid \neg A \mid A \land A \mid \xdianu \alpha\mid \xboxnuc\alpha &\quad\quad&\mathsf{S} \ni A::= p  \mid \top \mid \bot \mid \neg A \mid A \land A \mid  \alpha\mtra A
\\
\mathsf{N} \ni \alpha ::=  \dtop\mid \dbot \mid {\sim} \alpha \mid \alpha \dand \alpha \mid \dboxni A\mid \ddianni\alpha &\quad\quad&\mathsf{N} \ni \alpha ::=  \dtop\mid \dbot \mid {\sim} \alpha \mid \alpha \dand \alpha \mid \dboxni A\mid \drhdnni A.
\end{array}$
\end{center}
}
\paragraph{Algebraic semantics.} Interpretation of $\mathcal{L}_{MT\abla}$-formulas  (resp.~$\mathcal{L}_{MT>}$formulas) in heterogeneous m-algebras (resp.~c-algebras) under homomorphic assignments $h:\mathcal{L}_{MT\abla}\to \mathbb{H}$ (resp.~$h:\mathcal{L}_{MT>}\to \mathbb{H}$)  and validity of formulas in heterogeneous algebras ($\mathbb{H}\models\Theta$) are defined as usual. 

\paragraph{Frames and models.} $\mathcal{L}_{MT\abla}$-{\em models} (resp.~$\mathcal{L}_{MT>}$-{\em models}) are pairs $\mathbb{N} = (\mathbb{K}, V)$ s.t.~$\mathbb{K}= (X,Y,R_{\ni}, R_{\not\ni}, R_{\nu}, R_{\nu^c})$ is a supported two-sorted n-frame (resp.~$\mathbb{K}= (X,Y,R_{\ni},R_{\not\ni}, T_f)$ is a two-sorted c-frame) and $V:\mathcal{L}_{MT}\to\mathbb{K}^+$ is a heterogeneous algebra homomorphism of the appropriate signature. Hence, truth of formulas at states in models is defined as $\mb{N},z \Vdash \Theta$ iff $z\in V(\Theta)$ for every $z\in X\cup Y$ and $\Theta\in \mathsf{S}\cup\mathsf{N}$, and unravelling this stipulation   for formulas with a modal operator as main connective, we get:
{\small{
\begin{itemize}
\item $\mb{N},x \Vdash \xdianu \alpha \quad \text{iff}\quad  \mb{N},y \Vdash  \alpha \text{ for some } y \text{ s.t. } xR_\nu y$;
\item $\mb{N},x \Vdash \xboxnuc \alpha \quad \text{iff}\quad  \mb{N},y \Vdash  \alpha \text{ for all } y \text{ s.t. } xR_{\nu^c} y$;
\item $\mb{N},y \Vdash \dboxni A \quad \text{iff}\quad  \mb{N},x \Vdash  A \text{ for all } x \text{ s.t. } yR_\ni x$;
\item $\mb{N},y \Vdash \ddianni A \quad \text{iff}\quad  \mb{N},x \Vdash  A \text{ for some } x \text{ s.t. } yR_{\not\ni} x$;
\item $\mb{N},y \Vdash \drhdnni A \quad \text{iff}\quad  \mb{N},x \not\Vdash  A \text{ for all } x \text{ s.t. } yR_{\not\ni} x$;
\item $\mb{N},x \Vdash \alpha\mtra A \quad \text{iff}\quad  \text{ for all } y\text{ and all } x', \text{ if } T_f(x, y, x') \text{ and } \mb{N},y \Vdash \alpha \text{ then } \mb{N},x' \Vdash  A$.
\end{itemize}
}}
Global satisfaction (notation: $\mathbb{N}\Vdash\Theta$) is defined relative to the domain of the appropriate type, and frame validity (notation: $\mathbb{K}\Vdash\Theta$) is defined as usual. Thus, by definition, $\mathbb{K}\Vdash\Theta$ iff $\mathbb{K}^+\models \Theta$, 
and if $\mathbb{H}$ is a perfect heterogeneous algebra, then   $\mathbb{H}\models\Theta$ iff $\mathbb{H}_+\Vdash \Theta$. 
\paragraph{Sahlqvist theory for multi-type normal logics.} This semantic environment supports a straightforward extension of Sahlqvist theory for multi-type normal logics, which includes the definition of inductive and analytic inductive formulas and inequalities in $\mathcal{L}_{MT\abla}$  and $\mathcal{L}_{MT>}$ (cf.~Section \ref{sec:analytic inductive ineq}), and a corresponding version of the algorithm ALBA \cite{CoPa:non-dist} for computing their first-order correspondents and analytic structural rules. 
\paragraph{Translation.} Sahlqvist theory and analytic calculi for the non-normal logics  $\mathbf{L}_\abla$ and $\mathbf{L}_>$ and their analytic extensions can be then obtained `via translation', i.e.~by recursively defining  translations  $\tau_1, \tau_2:\mathcal{L}_\abla \to \mathcal{L}_{MT\abla}$ and $(\cdot)^\tau:\mathcal{L}_>\to \mathcal{L}_{MT>}$   as follows: 
{\small
\begin{center}
\begin{tabular}{rcl c rcl c rcl}
$\tau_1(p)$              &$=$& $p$                                        && $\tau_2(p)$              &$=$& $p$      &&         $p^\tau$ &=& $p$         \\
$\tau_1(\phi \xand \psi)$ &$=$& $\tau_1(\phi) \xand \tau_1(\psi)$ && $\tau_2(\phi \xand \psi)$ &$=$& $\tau_2(\phi) \xand \tau_2(\psi)$  &&  $(\phi \land \psi)^\tau$ &=& $\phi^\tau \land \psi^\tau$  \\
$\tau_1(\xneg \phi)$    &$=$& $\xneg \tau_2(\phi)$                 && $\tau_2(\xneg \phi)$    &$=$& $\xneg \tau_1(\phi)$    && $(\neg\phi)^\tau$ &=& $\neg \phi^\tau$                \\
$\tau_1(\abla \phi)$     &$=$& $\xdianu \dboxni \tau_1(\phi)$  && $\tau_2(\abla \phi)$     &$=$& $\xboxnuc \ddianni \tau_2(\phi)$ && $(\phi > \psi)^\tau$ &=& $(\dboxni \phi^\tau \wedge\drhdnni \phi^\tau)\mtra \psi^\tau$\\
\end{tabular}
\end{center}
 }
The following proposition is shown by a routine induction.
\begin{proposition}
If $\mathbb{F}$ is an n-frame (resp.~c-frame) and $\phi\vdash \psi$ is an $\mathcal{L}_{\abla}$-sequent  (resp.~$\phi$ is an $\mathcal{L}_{>}$-formula), then $\mathbb{F}\Vdash \phi\vdash \psi \quad\text{ iff }\quad \mathbb{F}^\star\Vdash \tau_1(\phi)\vdash \tau_2(\psi)$
(resp.~$\mathbb{F}\Vdash \phi\quad\text{ iff }\quad \mathbb{F}^\star\Vdash \phi^\tau$).
\end{proposition}
With this framework in place, we are in a position to (a) retrieve correspondence results in the setting of {\em non normal} logics, such as those collected in Theorem \ref{theor:correspondence-noAlba},
as instances of the general  Sahlqvist theory for multi-type {\em normal} logics, and (b) recognize whether the translation of a non normal  axiom  is analytic inductive, and compute its corresponding analytic structural rules (cf.~Section \ref{sec:ALBA runs}). 
{\small{
\begin{center}
\begin{tabular}{@{}r l c l cc}
&Axiom && Translation & Inductive & Analytic \\
N\, & $\abla \top\quad$ &&   $\top\leq \xboxnuc\ddianni \top$ & $\checkmark$ & $\checkmark$\\
P\, & $\neg \abla \bot\quad$ &&  $\top\leq \neg \xdianu \dboxni \bot$ & $\checkmark$ & $\checkmark$\\
C\, &$ \abla p \land \abla q \to \abla(p \land q)\quad$ &&  $\xdianu \dboxni  p \land\xdianu \dboxni q \leq \xboxnuc\ddianni (p \land q)$ & $\checkmark$ & $\checkmark$\\
T\, &$\abla p \to p \quad$ &&   $\xdianu \dboxni p\leq p $& $\checkmark$ & $\checkmark$\\
4\, & $\abla \abla  p \to \abla p\quad$ &&   $\xdianu \dboxni \xdianu \dboxni   p \leq\xboxnuc\ddianni  p$& $\checkmark$ & $\times$\\
4'\, & $\abla p \to \abla \abla p\quad$ &&  $\xdianu \dboxni  p \leq \xboxnuc\ddianni\xboxnuc\ddianni  p$& $\checkmark$ & $\times$\\
5\, & $\neg \abla \neg p \to \abla \neg \abla \neg p \quad$ &&  $\neg \xboxnuc\ddianni\neg  p \leq \xboxnuc\ddianni \neg\xdianu \dboxni \neg p$& $\checkmark$ & $\times$\\
B\, & $p \to \abla \neg \abla \neg p \quad$ &&   $  p \leq \xboxnuc\ddianni\neg \xdianu \dboxni \neg p$& $\checkmark$ & $\times$\\
D\, & $\abla p \to \neg \abla \neg p\quad$ &&   $\xdianu \dboxni   p \leq  \neg \xdianu \dboxni \neg p$& $\checkmark$ & $\checkmark$\\
CS\, & $(p\wedge q) \to (p > q)\quad$ &&  $(p\wedge q) \leq ((\dboxni p \wedge\drhdnni p)\mtra q)$& $\checkmark$ & $\checkmark$\\
CEM\, & $(p > q) \vee (p> \neg q)\quad$ &&  $\top\leq ((\dboxni p \wedge\drhdnni p)\mtra q) \vee ((\dboxni p \wedge\drhdnni p)\mtra \neg q)$& $\checkmark$ & $\checkmark$\\
ID\, & $p >  p \quad$ &&$\top\leq (\dboxni p \wedge\drhdnni p)\mtra p$& $\checkmark$ & $\checkmark$\\
\end{tabular}
\end{center}
}}

\section{Proper display calculi}

In this section we introduce proper multi-type display calculi for $\mathbf{L}_\abla$ and $\mathbf{L}_>$ and their axiomatic extensions generated by the analytic axioms in the table above. 

\noindent \emph{Languages. } \ \ The language $\mcl{L}_{DMT\abla}$ of the  calculus D.MT$\abla$ for $\mathbf{L}_\abla$ is defined as follows:
{\small{
\begin{center}
\begin{tabular}{l}
$\mathsf{S}\left\{\begin{array}{l}
A::= p \mid \xtop \mid \xbot \mid \neg A \mid A \xand A \mid \xdianu \alpha \mid \xboxnuc \alpha \\
X ::= A\mid \XTOP \mid \XBOT \mid \XNEG X \mid X \XAND X \mid X \XOR X \mid \XDIANU \Gamma \mid \XBOXNUC \Gamma \mid \XDIAIN \Gamma \mid \XBOXNIN \Gamma \\
\end{array} \right.$
 \\
$\mathsf{N}\left\{\begin{array}{l}
\alpha ::= \dboxni A \mid \ddianni A \\
\Gamma ::= \alpha \mid \DTOP \mid \DBOT \mid \DNEG \Gamma \mid \Gamma \DAND \Gamma \mid \Gamma \DOR \Gamma \mid \DBOXNI X \mid \DDIANNI X \mid \DBOXUN X \mid \DDIAUNC X \\
\end{array} \right.$
 \\
\end{tabular}
\end{center}
}}
The language $\mcl{L}_{DMT>}$ of the calculus D.MT$>$ for $\mathbf{L}_>$ is defined as follows:

{\small{
\begin{center}
\begin{tabular}{l}
$\mathsf{S}\left\{\begin{array}{l}
A::= p \mid \xtop \mid \xbot \mid \neg A \mid A \xand A \mid \alpha \mtra A \\
X ::= A\mid \XTOP \mid \XBOT \mid \XNEG X \mid X \XAND X \mid X \XOR X \mid \XDIAIN \Gamma \mid \Gamma \MTRA X \mid \Gamma \MTAND X \mid \XRHDNIN \Gamma\\
\end{array} \right.$
 \\
$\mathsf{N}\left\{\begin{array}{l}
\alpha ::= \dboxni A \mid \drhdnni A \mid \alpha \dand \alpha \\
\Gamma ::= \alpha \mid \DTOP \mid \DBOT \mid \DNEG \Gamma \mid \Gamma \DAND \Gamma \mid \Gamma \DOR \Gamma \mid \DBOXNI X \mid \DRHDNNI X \mid X \MTBRA X\\
\end{array} \right.$
 \\
\end{tabular}
\end{center}
}}

\noindent\emph{Multi-type display calculi.}\ \ In what follows, we use $X, Y, W, Z$ as structural $\mathsf{S}$-variables, and $\Gamma, \Delta, \Sigma, \Pi$ as structural $\mathsf{N}$-variables. 

\noindent {\bf Propositional base.}\ \ The calculi D.MT$\abla$ and D.MT$>$ share the rules listed below.
\begin{itemize}
\item Identity and Cut:
\end{itemize}
\vspace{-0.5cm}
{\fns
\begin{center}
\begin{tabular}{ccc}
\AXC{\rule[0mm]{0mm}{0.21cm}}
\LL{\scs $Id_\mathsf{S}$}
\UI$p \fCenter p$
\DP
 & 
\AX $X \fCenter A$
\AX $A \fCenter Y\rule[0mm]{0mm}{0.25cm}$
\RL{\scs $Cut_\mathsf{S}$}
\BI $X \fCenter Y$
\DP
 & 
\AX$\Gamma \fCenter \alpha$
\AX$\alpha \fCenter \Delta\rule[0mm]{0mm}{0.25cm}$
\RL{\scs $Cut_\mathsf{N}$}
\BI$\Gamma \fCenter \Delta$
\DP
\end{tabular}
\end{center}
}
\begin{itemize}
\item Pure $\mathsf{S}$-type display rules:
\end{itemize}
\vspace{-0.5cm}
{\fns
\begin{center}
\begin{tabular}{rlrlrl}
\AXC{\ }
\LL{\scs $\bot$}
\UI$\xbot \fCenter \XBOT$
\DP
 & 
\AXC{\ }
\RL{\scs $\top$}
\UI$\XTOP \fCenter \xtop$
\DP

 & 

\AX $X \XAND Y \fCenter Z$
\LeftLabel{\scriptsize $res_\mathsf{S}$}
\doubleLine
\UI $Y \fCenter \XNEG X \XOR Z$
\DisplayProof
 & 
\AX $X \fCenter Y \XOR Z $
\RightLabel{\scriptsize $res_\mathsf{S}$}
\doubleLine
\UI$\XNEG Y \XAND X \fCenter Z$
\DisplayProof

 & 

\AX $\XNEG X \fCenter Y$
\LeftLabel{\scriptsize $gal_\mathsf{S}$}
\doubleLine
\UI$\XNEG Y \fCenter X$
\DisplayProof
 & 
\AX$X \fCenter \XNEG Y$
\RightLabel{\scriptsize $gal_\mathsf{S}$}
\doubleLine
\UI$Y \fCenter \XNEG X$
\DisplayProof
 \\
\end{tabular}
\end{center}
}
\begin{itemize}
\item Pure $\mathsf{N}$-type display rules:
\end{itemize}
\vspace{-0.5cm}
{\fns
\begin{center}
\begin{tabular}{rlrl}
\AX$\Gamma \DAND \Delta \fCenter \Sigma$
\LeftLabel{\scriptsize $res_\mathsf{N}$}
\doubleLine
\UI$\Delta \fCenter \DNEG \Gamma \DOR \Sigma$
\DisplayProof
 & 
\AX$\Gamma \fCenter \Delta \DOR \Sigma$
\RightLabel{\scriptsize $res_\mathsf{N}$}
\doubleLine
\UI$\DNEG \Delta \DAND \Gamma \fCenter \Sigma$
\DisplayProof

\ \ & \ \ 

\AX$\DNEG \Gamma \fCenter \Delta$
\LeftLabel{\scriptsize $gal_\mathsf{N}$}
\doubleLine
\UI$\DNEG \Delta \fCenter \Gamma$
\DP
 & 
\AX$\Gamma \fCenter \DNEG \Delta$
\RL{\scriptsize $gal_\mathsf{N}$}
\doubleLine
\UI$\Delta \fCenter \DNEG \Gamma$
\DP
\end{tabular}
\end{center}
}
\begin{itemize}
\item Pure-type structural rules (these include standard Weakening (W), Contraction (C), Commutativity (E) and Associativity (A) in each type which we omit to save space): 
\end{itemize}
\vspace{-0.5cm}
{\fns
\begin{center}
\begin{tabular}{crlclrl}
\AX $X \fCenter  Y $
\LL{\scs $cont_\mathsf{S}$}
\doubleLine
\UI $\XNEG Y \fCenter \XNEG X$
\DP
 & 
\AX$X \fCenter Y$
\LL{\scs $\XTOP$}
\doubleLine
\UI$X \XAND \XTOP \fCenter Y$
\DisplayProof
 & 
\AX$X \fCenter  Y $
\RL{\scs $\XBOT$}
\doubleLine
\UI$X \fCenter Y \XOR \XBOT$
\DisplayProof

\ \ &\ \ 

\AX$\Gamma \fCenter \Delta$
\LL{\scs $cont_\mathsf{N}$}
\doubleLine
\UI$\DNEG \Delta \fCenter \DNEG \Gamma$
\DP
 & 
\AX $\Gamma \fCenter \Delta$
\LL{\scs $\DTOP$}
\doubleLine
\UI $\Gamma \DAND \DTOP \fCenter \Delta$
\DisplayProof
 & 
\AX$\Gamma \fCenter  \Delta $
\RL{\scs $\DBOT$}
\doubleLine
\UI$\Gamma \fCenter \Delta \DOR \DBOT$
\DP
 \\
\end{tabular}
\end{center}
}
\begin{itemize}
\item Pure  $\mathsf{S}$-type logical rules:
\end{itemize}
\vspace{-0.5cm}
{\fns
\begin{center}
\begin{tabular}{rlrl}
\AX$A \XAND B \fCenter X$
\LL{\scs $\xand$}
\UI$A \xand B \fCenter X$
\DP
 & 
\AX$X \fCenter A$
\AX$Y \fCenter B$
\RL{\scs $\xand$}
\BI$X \XAND Y \fCenter A \xand B$
\DP
 & 
\AX$\XNEG A \fCenter X$
\LL{\scs $\xneg$}
\UI$\xneg A \fCenter X$
\DP
 & 
\AX$X \fCenter \XNEG A$
\RL{\scs $\xneg$}
\UI$X \fCenter \xneg A$
\DP
 \\
\end{tabular}
\end{center}
}

\noindent {\bf Monotonic modal logic.} \ \ D.MT$\abla$ also includes the rules listed below.
\begin{itemize}
\item Multi-type display rules:
\end{itemize}
\vspace{-0.5cm}
{\fns
\begin{center}
\begin{tabular}{ccccc}
\AX$\XDIANU \Gamma \fCenter X$
\doubleLine
\LL{\scs $\XDIANU\DBOXUN$}
\UI$\Gamma \fCenter \DBOXUN X$
\DP
 & 
\AX$\DDIAUNC X \fCenter \Gamma$
\doubleLine
\LL{\scs $\DDIAUNC\XBOXNUC$}
\UI$X \fCenter \XBOXNUC \Gamma$
\DP
 & 
\AX$\XDIAIN \Gamma \fCenter X$
\doubleLine
\LL{\scs $\XDIAIN\DBOXNI$}
\UI$\Gamma \fCenter \DBOXNI X$
\DP
 & 
\AX$\XDIAIN \Gamma \fCenter X$
\doubleLine
\LL{\scs $\XDIAIN\DBOXNI$}
\UI$\Gamma \fCenter \DBOXNI X$
\DP
 & 
\AX$\DDIANNI X \fCenter \Gamma$
\doubleLine
\LL{\scs $\DDIANNI\XBOXNIN$}
\UI$X \fCenter \XBOXNIN \Gamma$
\DP
 \\
\end{tabular}
\end{center}
}
\begin{itemize}
\item Logical rules for multi-type connectives:
\end{itemize}
\vspace{-0.5cm}
{\fns
\begin{center}
\begin{tabular}{rlrlrl}
\AX$\XDIANU \alpha \fCenter X$
\LL{\scs $\xdianu$}
\UI$\xdianu \alpha \fCenter X$
\DP
 & 
\AX$\Gamma \fCenter \alpha$
\RL{\scs $\xdianu$}
\UI$\XDIANU \Gamma \fCenter \xdianu \alpha$
\DP
 & 
\AX$\alpha \fCenter \Gamma$
\LL{\scs $\xboxnuc$}
\UI$\xboxnuc \alpha \fCenter \XBOXNUC \Gamma$
\DP
 & 
\AX$X \fCenter \XBOXNUC \alpha$
\RL{\scs $\xboxnuc$}
\UI$X \fCenter \xboxnuc \alpha$
\DP
 \\

 & & & \\

\AX$\DDIANNI A \fCenter \Gamma$
\LL{\scs $\ddianni$}
\UI$\ddianni A \fCenter \Gamma$
\DP
 & 
\AX$X \fCenter A$
\RL{\scs $\ddianni$}
\UI$\DDIANNI X \fCenter \ddianni A$
\DP
 & 
\AX$A \fCenter X$
\LL{\scs $\dboxni$}
\UI$\dboxni A \fCenter \DBOXNI X$
\DP
 & 
\AX$\Gamma \fCenter \DBOXNI A$
\RL{\scs $\dboxni$}
\UI$\Gamma \fCenter \dboxni A$
\DP
 \\
\end{tabular}
\end{center}
}

\noindent {\bf Conditional logic.} \ \ D.MT$>$ includes left and right logical rules for $\dboxni$, the display postulates $\XDIAIN\DBOXNI$ and the rules listed below.
\begin{itemize}
\item Multi-type display rules:
\end{itemize}
\vspace{-0.5cm}
{\fns
\begin{center}
\begin{tabular}{ccc}
\AX$X \fCenter \Gamma \MTRA Y$
\doubleLine
\LL{\scs $\MTAND\MTRA$}
\UI$\Gamma \MTAND X \fCenter Y$
\doubleLine
\DP
\ \ & \ \ 
\AX$\Gamma \fCenter X \MTBRA Y$
\doubleLine
\RL{\scs $\MTBRA\MTRA$}
\UI$X \fCenter \Gamma \MTRA Y$
\DP
\ \ & \ \ 
\AX$X \fCenter \XRHDNIN \Gamma$
\doubleLine
\RL{\scs $\XRHDNIN\DRHDNNI$}
\UI$\Gamma \fCenter \DRHDNNI X$
\DP
 \\
\end{tabular}
\end{center}
}
\begin{itemize}
\item Logical rules for multi-type connectives and pure $\mathsf{G}$-type logical rules:
\end{itemize}
\vspace{-0.5cm}
{\fns
\begin{center}
\begin{tabular}{rlrlrl}
\AX$\Gamma \fCenter \alpha$
\AX$A \fCenter X$
\LL{\scs $\mtra$}
\BI$\alpha \mtra A \fCenter \Gamma \MTRA X$
\DP
 & 
\AX$X \fCenter \alpha \MTRA A$
\RL{\scs $\mtra$}
\UI$X \fCenter \alpha \mtra A$
\DP
 & 
\AX$X \fCenter A$
\LL{\scs $\drhdnni$}
\UI$\drhdnni A \fCenter \DRHDNNI X$
\DP
 & 
\AX$\Gamma \fCenter \DRHDNNI A$
\RL{\scs $\drhdnni$}
\UI$\Gamma \fCenter \drhdnni A$
\DP
 & 
\AX$\alpha \DAND \beta \fCenter \Gamma$
\LL{\scs $\dand$}
\UI$\alpha \dand \beta \fCenter \Gamma$
\DP
 & 
\AX$\Gamma \fCenter \alpha$
\AX$\Delta \fCenter \beta$
\RL{\scs $\dand$}
\BI$\Gamma \DAND \Delta \fCenter \alpha \dand \beta$
\DP
\\
\end{tabular}
\end{center}
}

\noindent {\bf Axiomatic extensions.} \ \ Each rule is labelled with the name of its corresponding axiom. 
\vspace{-0.5cm}
{\fns
\begin{center}
\begin{tabular}{ccc}
\AX$ \DDIANNI \XTOP \fCenter \Gamma$
\LL{\scs N}
\UI$ \XTOP \fCenter \XBOXNUC \Gamma$
\DP
 & 
\AX$\Delta \fCenter \DRHDNNI \XDIAIN \Gamma$
\AX$\XDIAIN \Gamma \fCenter X$
\LL{\scs ID}
\BI$\XTOP \fCenter (\Gamma\DAND \Delta)\MTRA X$
\DP
 & 
\AX$ \DDIANNI (\XDIAIN\Gamma\XAND\XDIAIN\Delta) \fCenter \Theta$
\LL{\scs C}
\UI$ \XDIANU\Gamma \XAND\XDIANU\Delta \fCenter \XBOXNUC \Theta $
\DP
 \\

 & & \\

\AX$\Gamma \fCenter \DBOXNI \XNEG\XDIAIN \Delta$
\LL{\scs D}
\UI$ \XDIANU \Delta \fCenter \XNEG\XDIANU \Gamma$
\DP
 & 
\AX$\Gamma \fCenter \DBOXNI\XBOT$
\LL{\scs P}
\UI$\XTOP \fCenter \XNEG\XDIANU \Gamma$
\DP
 & 
\AXC{$\Gamma \vdash \DBOXNI \XRHDNIN \Delta \quad\quad X \vdash \XRHDNIN \Delta \quad\quad Y \vdash Z$}
\LL{\scs CS}
\UIC{$X\XAND Y \vdash (\Gamma\DAND\Delta)\MTRA Z$}
\DP
 \\

 & & \\

\mc{3}{c}{
\AXC{$ \Pi\vdash \DRHDNNI \XDIAIN \Gamma  \quad  \Pi\vdash \DRHDNNI \XDIAIN \Theta  \quad  \Delta\vdash \DRHDNNI \XDIAIN \Gamma  \quad  \Delta\vdash \DRHDNNI \XDIAIN \Theta \quad Y\vdash X$}
\LL{\scs CEM}
\UIC{$ \XTOP\vdash ((\Gamma\DAND \Delta)\MTRA X)\XOR((\Theta\DAND \Pi)\MTRA \XNEG Y)$}
\DP
\ \ \ \ \ \ \ \ \ 
\AX$\Gamma \fCenter \DBOXNI X$
\LL{\scs T}
\UI$\XDIANU \Gamma \fCenter X$
\DP
}
\end{tabular}
\end{center}
}

\paragraph{Properties.} The calculi introduced above are proper (cf.~\cite{wansing2013displaying,greco2016unified}), and hence the  general theory of proper multi-type display calculi guarantees that they enjoy {\em cut elimination} and {\em subformula property} \cite{TrendsXIII}, and  are {\em sound} w.r.t.~their corresponding class of perfect heterogeneous algebras (or equivalently, two-sorted frames) \cite{greco2016unified}). In particular, key to the soundness argument for the axiomatic extensions is the observation that (multi-type) analytic inductive inequalities are canonical (i.e.~preserved under taking canonical extensions of heterogeneous algebras \cite{CoPa:non-dist}). Canonicity is also key to the proof of {\em conservativity} of the calculi w.r.t.~the original logics (this is a standard argument which is analogous to those in e.g.~\cite{greco2017multi,linearlogPdisplayed}). {\em Completeness} is argued by showing that the translations of each axiom is derivable in the corresponding calculus, and is sketched below.
{\fns 
\begin{itemize}
\item[N.] \ $ \abla \top \ \rightsquigarrow \ \xboxnuc\ddianni \top$ \ \ \ \ P. \ $\neg \abla \bot \ \rightsquigarrow\ \neg \xdianu \dboxni \bot$ \ \ \ \ T. \ $\abla A \to A \ \rightsquigarrow\ \xdianu \dboxni A \vdash A$
\end{itemize}
\vspace{-0.5cm}
{\fns
\begin{center}
\begin{tabular}{ccc}
\AX$\XTOP \fCenter \xtop$
\UI$\DDIANI \XTOP \fCenter \ddiani \xtop$
\LL{\scs N}
\UI$\XTOP \fCenter \XBOXNUC \ddiani \xtop$
\DP
\ \ & \ \ 
\AX$\xbot \fCenter \XBOT$
\UI$\dboxni \xbot \fCenter \DBOXNI \XBOT$
\LL{\scs P}
\UI$\XTOP \fCenter \XNEG \dboxni \xbot$
\DP
\ \ & \ \ 
\AX$A \fCenter A$
\UI$\dboxni A \fCenter \DBOXNI A$
\LL{\scs T}
\UI$\XDIANU \dboxni A \fCenter A$
\DP
 \\
\end{tabular}
\end{center}
}
\begin{itemize}
\item[ID.] $A > A \ \rightsquigarrow\ (\dboxni A \wedge\drhdnni A) \mtra A$ \ \ \ \ \ \ \ CS.\ $(A \wedge B) \to (A > B) \ \rightsquigarrow\ (A \wedge B) \vdash (\dboxni A \dand \drhdnni A) \mtra B$
\end{itemize}
\vspace{-0.5cm}
{\fns
\begin{center}
\begin{tabular}{@{}cc}
\AX$A \fCenter A$
\UI$\drhdnni A \fCenter \DRHDNNI A$
\UI$A \fCenter \XRHDNIN \drhdnni A$
\UI$\dboxni A \fCenter \DBOXNI \XRHDNIN \drhdnni A$
\UI$\XDIAIN \dboxni A \fCenter \XRHDNIN \drhdnni A$
\UI$\drhdnni A \fCenter \DRHDNNI \XDIAIN \dboxni A$
\AX$A \fCenter A$
\UI$\dboxni A \fCenter \DBOXNI A$
\UI$\XDIAIN \dboxni A \fCenter A$
\LL{\scs ID}
\BI$\XTOP \fCenter (\DBOXNI A \DAND \DRHDNNI A) \MTRA A$
\DP

 & 

\AX$A \fCenter A$
\UI$\drhdnni A \fCenter \DRHDNNI A$
\UI$A \fCenter \XRHDNIN \drhdnni A$
\UI$\dboxni A \fCenter \DBOXNI \XRHDNIN \drhdnni A$
\AX$A \fCenter A$
\UI$\drhdnni A \fCenter \DRHDNNI A$
\UI$A \fCenter \XRHDNIN \drhdnni A$
\AX$B \fCenter B$
\LL{\scs CS}
\TI$A \XAND B \fCenter (\DBOXNI A \DAND \DRHDNNI A) \MTRA B$
\DP 
 \\
\end{tabular}
\end{center}
}

\begin{itemize}
\item[CEM.] $(A > B) \vee (A > \neg B) \ \rightsquigarrow\ (\dboxni A \dand \drhdnni A)\mtra B \vee (\dboxni A \dand \drhdnni A) \mtra \neg B$
\end{itemize}
{\fns
\begin{center}
\AX$\drhdnni A \fCenter \DRHDNNI \XDIAIN \dboxni A \ \ \ \drhdnni A \fCenter \DRHDNNI \XDIAIN \dboxni A\ \ \ \drhdnni A \fCenter \DRHDNNI \XDIAIN \dboxni A\ \ \ \drhdnni A \fCenter \DRHDNNI \XDIAIN \dboxni A$
\LL{\scs CEM}
\UI$\XTOP \fCenter (\dboxni A \DAND \drhdnni A) \MTRA B \XOR (\dboxni A \DAND \drhdnni A) \MTRA \XNEG B$
\DP
\end{center}
}

\begin{itemize}
\item[C.] $ \abla A \land \abla B \to \abla(A \land B) \rightsquigarrow \xdianu \dboxni A \land\xdianu \dboxni B \vdash \xboxnuc\ddianni (A \land B)$ 
\ \ D. \ \mbox{$\abla A \to \neg \abla \neg A \rightsquigarrow \xdianu \dboxni A \vdash \neg \xdianu \dboxni \neg A$}
\end{itemize}
{\fns
\begin{center}
\begin{tabular}{cc}

\AX$A \fCenter A$
\UI$\dboxni A \fCenter \dboxni A$
\UI$\XDIAIN \dboxni A \fCenter A$
\AX$B \fCenter B$
\UI$\dboxni B \fCenter \dboxni B$
\UI$\XDIAIN \dboxni B \fCenter B$
\BI$\XDIAIN \dboxni A \XAND \XDIAIN \dboxni B \fCenter A \xand B$
\UI$\DDIANNI (\XDIAIN \dboxni A \XAND \XDIAIN \dboxni B) \fCenter \ddianni (A \xand B)$
\LL{\scs C}
\UI$\XDIANU \dboxni A \XAND \XDIANU \dboxni B \fCenter \XBOXNUC \ddianni (A \xand B)$
\DP
 & 
\AX$A \fCenter A$
\UI$\dboxni A \fCenter \DBOXNI A$
\UI$\XDIAIN \dboxni A \fCenter A$
\UI$\xneg A \fCenter \XNEG \XDIAIN \dboxni A$
\UI$\dboxni \xneg A \fCenter \DBOXNI \XNEG \XDIAIN \dboxni A$
\LL{\scs D}
\UI$\XDIANU \dboxni A \fCenter \XNEG \XDIANU \dboxni \xneg A$
\DP

 \\
\end{tabular}
\end{center}
}

}


\appendix
\section{Analytic inductive inequalities}
\label{sec:analytic inductive ineq}
In the present section, we  specialize the definition of {\em analytic inductive inequalities} (cf.\ \cite{greco2016unified}) to the  multi-type languages $\mathcal{L}_{MT\abla}$ and $\mathcal{L}_{MT>}$ reported below.
{\small
\begin{center}
$\begin{array}{lll}
\mathsf{S} \ni A::= p  \mid \top \mid \bot \mid \neg A \mid A \land A \mid \xdianu \alpha\mid \xboxnuc\alpha &\quad\quad&\mathsf{S} \ni A::= p  \mid \top \mid \bot \mid \neg A \mid A \land A \mid  \alpha\mtra A
\\
\mathsf{N} \ni \alpha ::=  \dtop\mid \dbot \mid {\sim} \alpha \mid \alpha \dand \alpha \mid \dboxni A\mid \ddianni A &\quad\quad&\mathsf{N} \ni \alpha ::=  \dtop\mid \dbot \mid {\sim} \alpha \mid \alpha \dand \alpha \mid \dboxni A\mid \drhdnni A.
\end{array}$
\end{center}
}
An {\em order-type} over $n\in \mathbb{N}$  is an $n$-tuple $\epsilon\in \{1, \partial\}^n$. If $\epsilon$ is an order type, $\epsilon^\partial$ is its {\em opposite} order type; i.e.~$\epsilon^\partial(i) = 1$ iff $\epsilon(i)=\partial$ for every $1 \leq i \leq n$.
The connectives of the language above are grouped together  into the  families $\mathcal{F}: = \mathcal{F}_{\mathsf{S}}\cup \mathcal{F}_{\mathsf{N}}\cup \mathcal{F}_{\textrm{MT}}$ and $\mathcal{G}: = \mathcal{G}_{\mathsf{S}}\cup \mathcal{G}_{\mathsf{N}} \cup  \mathcal{G}_{\textrm{MT}}$, defined as follows:
\begin{center}
\begin{tabular}{lcl}
$\mathcal{F}_{\mathsf{S}}: = \{\xneg\}$&&$ \mathcal{G}_{\mathsf{S}} = \{\xneg\}$\\
$\mathcal{F}_{\mathsf{N}}: = \{\dneg\}$ && $\mathcal{G}_{\mathsf{N}}: = \{\dneg\}$\\
$\mathcal{F}_{\textrm{MT}}: = \{\xdianu, \ddianni \}$ && $\mathcal{G}_{\textrm{MT}}: = \{\dboxni, \xboxnuc, \mtra, \drhdnni\}$\\

\end{tabular}
\end{center}
For any $f\in \mathcal{F}$  (resp.\ $g\in \mathcal{G}$), we let $n_f\in \mathbb{N}$ (resp.~$n_g\in \mathbb{N}$) denote the arity of $f$ (resp.~$g$), and the order-type $\epsilon_f$ (resp.~$\epsilon_g$) on $n_f$ (resp.~$n_g$)  indicate whether the $i$th coordinate of $f$ (resp.\ $g$) is positive ($\epsilon_f(i) = 1$,  $\epsilon_g(i) = 1$) or  negative ($\epsilon_f(i) = \partial$,  $\epsilon_g(i) = \partial$). 
%
				%
				\begin{definition}[\textbf{Signed Generation Tree}]
					\label{def: signed gen tree}
					The \emph{positive} (resp.\ \emph{negative}) {\em generation tree} of any $\mathcal{L}_\textrm{MT}$-term $s$ is defined by labelling the root node of the generation tree of $s$ with the sign $+$ (resp.\ $-$), and then propagating the labelling on each remaining node as follows:
					For any node labelled with $\ell\in \mathcal{F}\cup \mathcal{G}$ of arity $n_\ell$, and for any $1\leq i\leq n_\ell$, assign the same (resp.\ the opposite) sign to its $i$th child node if $\epsilon_\ell(i) = 1$ (resp.\ if $\epsilon_\ell(i) = \partial$). Nodes in signed generation trees are \emph{positive} (resp.\ \emph{negative}) if are signed $+$ (resp.\ $-$).
				\end{definition}
				For any term $s(p_1,\ldots p_n)$, any order type $\epsilon$ over $n$, and any $1 \leq i \leq n$, an \emph{$\epsilon$-critical node} in a signed generation tree of $s$ is a leaf node $+p_i$ with $\epsilon(i) = 1$ or $-p_i$ with $\epsilon(i) = \partial$. An $\epsilon$-{\em critical branch} in the tree is a branch ending in an $\epsilon$-critical node. For any term $s(p_1,\ldots p_n)$ and any order type $\epsilon$ over $n$, we say that $+s$ (resp.\ $-s$) {\em agrees with} $\epsilon$, and write $\epsilon(+s)$ (resp.\ $\epsilon(-s)$), if every leaf in the signed generation tree of $+s$ (resp.\ $-s$) is $\epsilon$-critical.
				 We will also write $+s'\prec \ast s$ (resp.\ $-s'\prec \ast s$) to indicate that the subterm $s'$ inherits the positive (resp.\ negative) sign from the signed generation tree $\ast s$. Finally, we will write $\epsilon(s') \prec \ast s$ (resp.\ $\epsilon^\partial(s') \prec \ast s$) to indicate that the signed subtree $s'$, with the sign inherited from $\ast s$, agrees with $\epsilon$ (resp.\ with $\epsilon^\partial$).	
		\begin{definition}[\textbf{Good branch}]
					\label{def:good:branch}
					Nodes in signed generation trees will be called \emph{$\Delta$-adjoints}, \emph{syntactically left residual (SLR)}, \emph{syntactically right residual (SRR)}, and \emph{syntactically right adjoint (SRA)}, according to the specification given in Table \ref{Join:and:Meet:Friendly:Table}.
					A branch in a signed generation tree $\ast s$, with $\ast \in \{+, - \}$, is called a \emph{good branch} if it is the concatenation of two paths $P_1$ and $P_2$, one of which may possibly be of length $0$, such that $P_1$ is a path from the leaf consisting (apart from variable nodes) only of PIA-nodes 
					and $P_2$ consists (apart from variable nodes) only of Skeleton-nodes.

\vspace{-0.3cm}

\begin{table}
\begin{tabular}{cc}

							\begin{tabular}{| c | c |}
								\hline
								Skeleton &PIA\\
								\hline
								$\Delta$-adjoints & SRA \\
								\begin{tabular}{ c c c c c c c}
									$+\ $ &\ $\xor\ $ &\ $\dor\ $ \\ 
									$-\ $ & $\xand$ &$\dand$ \\ 
								\end{tabular}
								&
								\begin{tabular}{c c c c c c c c c  }
									$+\ $ & \ $\xand$\  &\ $\dand$\ &\ $\dboxni$ \ & \ $\xboxnuc$ \ & \ $\mtra$ \ & \ $\drhdnni$\ &\ $\rdneg$\ &\ $\dneg$\ \\
									$-\ $ & \  $\xor $\  &\ $\dor$\  &\ $\xdianu$ \ &\ $ \ddianni$\ & \ $\rdneg$ \ &\ $\dneg$\    \\
								\end{tabular}
								\\
								\hline
								SLR &SRR\\
								\begin{tabular}{c c c c c c c c c}
									$+\ $ & \ $\xand$\  &\ $\dand$\ &\ $\xdianu$ \ &\ $ \ddianni$\ & \ $\rdneg$ \ &\ $\dneg$\   \\
									$-\ $ & \ $\xor$\  &\ $\dor$\  &\ $\dboxni$ \ & \ $\xboxnuc$ \ & \ $\mtra$ \ & \ $\drhdnni$\ &\ $\rdneg$\ &\ $\dneg$\  \\
								\end{tabular}
								&\begin{tabular}{c c c}
									 $+\ $ &\ $\xor$\  &\ $\dor$\ \\
									$-\ $ &\ $\xand$ \ &\ $\dand$ \ \\
								\end{tabular}
								\\
								\hline

							\end{tabular}

 & 

\begin{tabular}{c}
		\begin{tikzpicture}[scale=0.4]
		\draw (-5,-1.5) -- (-3,1.5) node[above]{\Large$+$} ;
		\draw (-5,-1.5) -- (-1,-1.5) ;
		\draw (-3,1.5) -- (-1,-1.5);
		\draw (-6,0) node{Skeleton} ;
		\draw[dashed] (-3,1.5) -- (-4,-1.5);
		\draw[dashed] (-3,1.5) -- (-2,-1.5);
		\draw (-4,-1.5) --(-4.8,-3);
		\draw (-4.8,-3) --(-3.2,-3);
		\draw (-3.2,-3) --(-4,-1.5);
		\draw[dashed] (-4,-1.5) -- (-4,-3);
		\draw[fill] (-4,-3) circle[radius=.1] node[below]{$+p$};
		\draw
		(-2,-1.5) -- (-2.8,-3) -- (-1.2,-3) -- (-2,-1.5);
		\fill[pattern=north east lines]
		(-2,-1.5) -- (-2.8,-3) -- (-1.2,-3);
		\draw (-2,-3.5)node{$s_1$};
		\draw (-6,-2.25) node{PIA} ;
		\draw (0,1.8) node{$\leq$};
		\draw (5,-1.5) -- (3,1.5) node[above]{\Large$-$} ;
		\draw (5,-1.5) -- (1,-1.5) ;
		\draw (3,1.5) -- (1,-1.5);
		\draw (6,0) node{Skeleton} ;
		\draw[dashed] (3,1.5) -- (4,-1.5);
		\draw[dashed] (3,1.5) -- (2,-1.5);
		\draw (2,-1.5) --(2.8,-3);
		\draw (2.8,-3) --(1.2,-3);
		\draw (1.2,-3) --(2,-1.5);
		\draw[dashed] (2,-1.5) -- (2,-3);
		\draw[fill] (2,-3) circle[radius=.1] node[below]{$+p$};
		\draw
		(4,-1.5) -- (4.8,-3) -- (3.2,-3) -- (4, -1.5);
		\fill[pattern=north east lines]
		(4,-1.5) -- (4.8,-3) -- (3.2,-3) -- (4, -1.5);
		\draw (4,-3.5)node{$s_2$};
		\draw (6,-2.25) node{PIA} ;
		\end{tikzpicture}
		\\
\end{tabular}

 \\
\end{tabular}
\vspace{0.5em}	
\caption{Skeleton and PIA nodes.}\label{Join:and:Meet:Friendly:Table}
\end{table}

\end{definition}

\vspace{-1cm}	
				
				\begin{definition}[\textbf{Analytic inductive inequalities}]
	\label{def:analytic inductive ineq}
					For any order type $\epsilon$ and any irreflexive and transitive relation $<_\Omega$ on $p_1,\ldots p_n$, the signed generation tree $*s$ $(* \in \{-, + \})$ of an $\mathcal{L}_{MT}$ term $s(p_1,\ldots p_n)$ is \emph{analytic $(\Omega, \epsilon)$-inductive} if
					\begin{enumerate}
						\item  every branch of $*s$ is good (cf.\ Definition \ref{def:good:branch});
						\item for all $1 \leq i \leq n$, every SRR-node occurring in  any $\epsilon$-critical branch with leaf $p_i$ is of the form $ \circledast(s, \beta)$ or $ \circledast(\beta, s)$, where the critical branch goes through $\beta$ and 
						\begin{enumerate}
							\item $\epsilon^\partial(s) \prec \ast s$ (cf.\ discussion before Definition \ref{def:good:branch}), and
							%
							\item $p_k <_{\Omega} p_i$ for every $p_k$ occurring in $s$ and for every $1\leq k\leq n$.
						\end{enumerate}

					\end{enumerate}
					
					 An inequality $s \leq t$ is \emph{analytic $(\Omega, \epsilon)$-inductive} if the signed generation trees $+s$ and $-t$ are analytic $(\Omega, \epsilon)$-inductive. An inequality $s \leq t$ is \emph{analytic inductive} if is analytic $(\Omega, \epsilon)$-inductive for some $\Omega$ and $\epsilon$.
				\end{definition}
								

\section{Algorithmic proof of Theorem \ref{theor:correspondence-noAlba}}
\label{sec:ALBA runs}

In what follows, we show that the correspondence results collected in Theorem \ref{theor:correspondence-noAlba} can be retrieved as instances of a suitable multi-type version of algorithmic correspondence  for normal logics (cf.~\cite{CoGhPa14,CoPa:non-dist}), hinging on the usual order-theoretic properties of the algebraic interpretations of the logical connectives, while admitting nominal variables of two sorts. For the sake of enabling a swift translation into the language of m-frames and c-frames, we write nominals directly as singletons, and, abusing notation, we quantify over the elements defining these singletons. These computations also serve to prove that each analytic structural rule is sound on the heterogeneous perfect algebras validating its correspondent axiom. In the computations relative to each analytic axiom, the line marked with $(\star)$ marks the quasi-inequality that interprets the corresponding analytic rule. This computation does {\em not} prove the equivalence between the axiom and the rule, since the variables occurring in each starred quasi-inequality are restricted rather than arbitrary. However, the proof of soundness is completed by observing that all ALBA rules in the steps above the marked inequalities are (inverse) Ackermann and adjunction rules, and hence are sound also when arbitrary variables replace (co-)nominal variables. 

\vspace{-0.3cm}

{\small{
\begin{flushleft}
\begin{tabular}{@{}cll | cll}
\mc{3}{l}{N.\ \, $\mathbb{F}\Vdash \nabla \top \ \rightsquigarrow\ \top\subseteq [\nu^c] \langle \not \ni \rangle \top$} \ & \mc{3}{l}{\ P. \ $\mathbb{F}\models \neg \nabla \bot\ \rightsquigarrow\ \top\subseteq \neg \langle \nu \rangle [\ni ] \bot$} \\
\hline
     & $\top \subseteq [\nu^c] \langle \not \ni \rangle \top $ & \ & \ & $\top \subseteq\neg \langle \nu \rangle [\ni ] \bot$  & \\
iff  & $\forall X \forall w [\langle \not \ni \rangle \top \subseteq \{X\}^c  \Rightarrow \{w\} \subseteq [\nu^c] \{X\}^c]$  & $(\star)$ first app. \ & \ iff & $ \forall X [ X \subseteq [\ni]\bot \Rightarrow T \subseteq \neg \langle \nu \rangle  X]$ & $(\star)$ first app. \\
iff  & $\forall X \forall w[X = W  \Rightarrow \{w\} \subseteq [\nu^c] \{X\}^c)$  &  ($\langle \ni \rangle \top =  \{W\}^c$) \ & \ iff  & $W \subseteq\neg \langle \nu \rangle [\ni ] \emptyset$  & \\
iff  & $\forall w[ \{w\} \subseteq [\nu^c] \{W\}^c]$  & \ & \ iff  & $W \subseteq\neg \langle \nu \rangle \{ \emptyset \}$  & $[\ni ] \emptyset = \{Z\subseteq W\mid Z\subseteq \emptyset\}$ \\
iff  & $\forall w[\{w\} \subseteq (R_{\nu^c}^{-1}[W])^c]$  & \ & \ iff  & $W \subseteq \{w \in W \mid  w R_\nu \emptyset\}^c$  & \\
iff  & $\forall w[ \{w\} \subseteq R_{\nu}^{-1}[W]]$  & \ & \ iff  & $\forall w[ \emptyset \not \in \nu(w)]$.  & \\
iff  & $\forall w[ W\in \nu (w) ]$  & & & & \\
\end{tabular}
\end{flushleft}
 }

{\small{
\begin{flushleft}
\begin{tabular}{c ll}
\mc{3}{l}{C.\ \, $\mathbb{F}\models \nabla p \land \nabla q \to \nabla ( p \land q)  \ \rightsquigarrow\ \langle \nu \rangle [\ni] p \land \langle \nu \rangle [\ni] q \subseteq [\nu^c] \langle \not \ni \rangle (p \land q)$} \\
\hline
&$\langle \nu \rangle [\ni] p \land \langle \nu \rangle [\ni] q \subseteq [\nu^c] \langle \not \ni \rangle (p \land q)$\\
iff & $\forall  Z_1 \forall  Z_2 \forall Z_3\forall p \forall q[ \{Z_1\}\subseteq [\ni] p \ \& \   \{Z_2\} \subseteq [\ni] q \ \&\  \langle \not \ni \rangle(p \land q)\subseteq  \{Z_3\}^c\Rightarrow \langle \nu \rangle  \{Z_1\} \land \langle \nu \rangle  \{Z_2\} \subseteq [\nu^c]\{Z_3\}^c]$ & first approx. \\
iff & $\forall  Z_1 \forall  Z_2 \forall Z_3\forall p \forall q[ \langle \in \rangle \{Z_1\}\subseteq p \ \& \ \langle \in \rangle   \{Z_2\} \subseteq  q \ \&\  \langle \not \ni \rangle(p \land q)\subseteq  \{Z_3\}^c \Rightarrow \langle \nu \rangle  \{Z_1\} \land \langle \nu \rangle  \{Z_2\} \subseteq [\nu^c] \{Z_3\}^c]$ & Residuation \\
iff & $\forall  Z_1 \forall  Z_2 \forall Z_3[   \langle \not \ni \rangle(\langle \in \rangle \{Z_1\} \land \langle \in \rangle   \{Z_2\} )\subseteq  \{Z_3\}^c\Rightarrow \langle \nu \rangle  \{Z_1\} \land \langle \nu \rangle  \{Z_2\} \subseteq [\nu^c] \{Z_3\}^c]$ &$(\star)$ Ackermann \\
iff & $\forall  Z_1 \forall  Z_2 \forall Z_3[   (\langle \in \rangle \{Z_1\} \land \langle \in \rangle   \{Z_2\} )\subseteq [\not \in] \{Z_3\}^c\Rightarrow \langle \nu \rangle  \{Z_1\} \land \langle \nu \rangle  \{Z_2\} \subseteq [\nu^c] \{Z_3\}^c]$ & Residuation \\
iff &$\forall  Z_1 \forall  Z_2 \forall Z_3[\forall x(xR_\in Z_1\ \& \ xR_\in Z_2\Rightarrow \lnot  xR_{\notin}Z_3)\Rightarrow\forall x(xR_\nu Z_1\ \&\ xR_\nu Z_2\Rightarrow \lnot xR_{\nu^c}Z_3)]$ & Standard translation \\
iff &$\forall  Z_1 \forall  Z_2 \forall Z_3[\forall x(x\in Z_1\ \& \ x\in Z_2\Rightarrow x\in Z_3)\Rightarrow\forall x(Z_1\in \nu(x)\ \& \ Z_2\in\nu(x)\Rightarrow Z_3\in\nu(x))]$ & Relations interpretation \\
iff &$\forall  Z_1 \forall  Z_2 \forall Z_3[Z_1\cap Z_2\subseteq Z_3\Rightarrow\forall x(Z_1\in \nu(x)\ \& \ Z_2\in\nu(x)\Rightarrow Z_3\in\nu(x))]$ &  \\
iff &$\forall  Z_1 \forall  Z_2 \forall x(Z_1\in \nu(x)\ \& \ Z_2\in\nu(x)\Rightarrow Z_1\cap Z_2\in\nu(x))]$. & Monotonicity  \\
\end{tabular}
\end{flushleft}
}}

{\small{
\begin{flushleft}
\begin{tabular}{c ll}
\mc{3}{l}{T.\ \, $\mathbb{F}\models \nabla p\to p\ \rightsquigarrow\ \langle \nu\rangle [\ni] p\subseteq p$} \\
\hline
&$\langle \nu\rangle [\ni] p\subseteq p$\\
iff & $\forall  x\forall Z\forall p [ p \subseteq \{x\}^c\ \&\ \{Z\}\subseteq\dboxni p \Rightarrow \langle \nu\rangle\{Z\} \subseteq  \{x\}^c]$ & first approx. \\
iff & $\forall  x\forall Z\forall p [ p \subseteq \{x\}^c\ \&\ \langle\in\rangle\{Z\}\subseteq p \Rightarrow \langle \nu\rangle\{Z\} \subseteq  \{x\}^c]$ & Adjunction \\
iff & $\forall  x\forall Z[\langle\in\rangle\{Z\}\subseteq \{x\}^c \Rightarrow \langle \nu\rangle\{Z\} \subseteq  \{x\}^c]$ & $(\star)$ Ackermann \\
iff & $\forall  Z  [\langle \nu\rangle \{Z\} \subseteq \langle \ni \rangle \{Z\} ]$ & inverse approx. \\
iff  & $\forall x\forall Z[x  R_\nu Z\Rightarrow xR_\ni Z]$ &  Standard translation\\
iff  & $\forall x\forall Z[Z\in \nu(x)\Rightarrow x \in Z]$. & Relation translation\\
\end{tabular}
\end{flushleft}
}}

{\small{
\begin{center}
\begin{tabular}{c ll}
\mc{3}{l}{4'.\ \, $\mathbb{F}\models \nabla p\to \nabla \nabla p\ \rightsquigarrow\ \langle \nu\rangle [\ni] p\subseteq [\nu^c] \langle \not \ni \rangle [\nu^c] \langle \not \ni \rangle  p$} \\
\hline
&$\langle \nu\rangle [\ni] p\subseteq [\nu^c] \langle \not \ni \rangle [\nu^c] \langle \not \ni \rangle  p$\\ 
iff & $\forall Z_1\forall x' \forall p [ \{Z_1\}  \subseteq [\ni] p \ \&\ [\nu^c] \langle \not \ni \rangle [\nu^c] \langle \not \ni \rangle p\subseteq \{x'\}^c)\Rightarrow \langle \nu \rangle  \{Z_1\}  \subseteq  \{x'\}^c]$ & first approx. \\
iff & $\forall  Z_1 \forall x'\forall p [\langle \in \rangle  \{Z_1\}  \subseteq  p \ \&\ [\nu^c] \langle \not \ni \rangle [\nu^c] \langle \not \ni \rangle p\subseteq  \{x'\}^c)\Rightarrow \langle \nu \rangle  \{Z_1\}  \subseteq  \{x'\}^c]$ & Residuation \\
iff & $\forall  Z_1 \forall x'[[\nu^c] \langle \not \ni \rangle [\nu^c] \langle \not \ni \rangle\langle \in \rangle \{Z_1\}  \subseteq  \{x'\}^c \Rightarrow \langle \nu \rangle  \{Z_1\}  \subseteq  \{x'\}^c]$ & Ackermann \\
iff & $\forall  Z_1[\langle \nu \rangle  \{Z_1\}   \subseteq[\nu^c] \langle \not \ni \rangle [\nu^c] \langle \not \ni \rangle\langle \in \rangle \{Z_1\} ]$ &  \\
iff & $\forall  Z_1 \forall x[ xR_\nu Z_1    \Rightarrow \forall Z_2 (x R_{\nu^c} Z_2 \Rightarrow \exists y (Z_2 R_{\not \ni} y \ \& \ \forall Z_3 (yR_{\nu^c} Z_3 \Rightarrow \exists w ( Z_3 R_{\not \ni} w \ \& \ wR_\in Z_1 ) )))]$ & Standard translation\\ 
iff & $\forall  Z_1 \forall x[ x \in \nu(Z)   \Rightarrow \forall Z_2 (Z_2 \not \in \nu(x) \Rightarrow \exists y (y \not \in Z_2 \ \& \ \forall Z_3 (Z_2 \not \in \nu(y) \Rightarrow \exists w( w \not \in Z_3 \ \& \ w \in Z_1 ))))]$ & Relations translation\\
iff & $\forall  Z_1 \forall x[ x \in \nu(Z)   \Rightarrow \forall Z_2 (Z_2 \not \in \nu(x) \Rightarrow \exists y (y \not \in Z_2 \ \& \ \forall Z_3 (Z_2 \not \in \nu(y) \Rightarrow Z_1\nsubseteq Z_3)))]$ & Relations translation\\
iff &$\forall  Z_1 \forall x[ x \in \nu(Z)   \Rightarrow(\forall Z_2(\forall y(\forall Z_3(Z_1\subseteq Z_3\Rightarrow Z_3\in\nu(y))\Rightarrow y\in Z_2)\Rightarrow Z_2\in\nu(x)))]$ & Contraposition\\
iff &$\forall  Z_1 \forall x[ x \in \nu(Z) \Rightarrow(\forall Z_2(\forall y(Z_1\in\nu(y))\Rightarrow y\in Z_2)\Rightarrow Z_2\in\nu(x)))]$ & Monotonicity\\
iff &$\forall  Z_1 \forall x[ x \in \nu(Z) \Rightarrow\{ y\mid Z_1\in\nu(y)\}\in\nu(x)]$. & Monotonicity
\end{tabular}
\end{center}
}}

{\small{ 
\begin{flushleft}
\begin{tabular}{c ll}
\mc{3}{l}{4.\ \, $\mathbb{F}\models \nabla \nabla  p\to \nabla p\ \rightsquigarrow\ \langle \nu\rangle [\ni] \langle \nu\rangle [\ni] p\subseteq[ \nu^c ] \langle \not \ni \rangle p$} \\
\hline
&$ \langle \nu\rangle [\ni]\langle \nu\rangle [\ni] p\subseteq[ \nu^c ] \langle \not \ni \rangle p$\\ 
iff & $\forall  x\forall Z_1 \forall p[ \{x\} \subseteq \xdianu\dboxni\xdianu\dboxni p\ \&\ \ddianni p\subseteq\{Z_1\}^c \Rightarrow  \{x\}\subseteq \xboxnuc\{Z_1\}^c]$ & first approx. \\
iff & $\forall  x\forall Z_1 \forall p[ \{x\} \subseteq \xdianu\dboxni\xdianu\dboxni p\ \&\  p\subseteq [\notin]\{Z_1\}^c \Rightarrow  \{x\}\subseteq \xboxnuc\{Z_1\}^c]$ & Adjunction \\
iff & $\forall  x\forall Z_1[ \{x\} \subseteq \xdianu\dboxni\xdianu\dboxni  [\notin]\{Z_1\}^c \Rightarrow  \{x\}\subseteq \xboxnuc\{Z_1\}^c]$ & Ackermann \\
iff &$\forall  x\forall Z_1[(\exists Z_2(xR_\nu Z_2\ \&\ \forall y(Z_2 R_\ni y\Rightarrow\exists Z_3(y R_\nu Z_3\ \&\ \forall w(Z_3 R_\ni w\Rightarrow \lnot wR_{\not\in} Z_1)))))\Rightarrow \lnot x R_{\nu^c} Z_1]$ & Standard translation\\
iff &$\forall  x\forall Z_1[((\exists Z_2\in \nu(x))(\forall y\in Z_2)(\exists Z_3\in\nu(y))(\forall w\in Z_3)(w\in Z_1))\Rightarrow Z_1\in\nu(x)]$ & Relation translation\\
iff &$\forall  x\forall Z_1[((\exists Z_2\in \nu(x))(\forall y\in Z_2)(\exists Z_3\in\nu(y))(Z_3\subseteq Z_1))\Rightarrow Z_1\in\nu(x)]$ & \\
iff &$\forall  x\forall Z_1\forall Z_2[(Z_2\in \nu(x)\ \&\ (\forall y\in Z_2)(\exists Z_3\in\nu(y))(Z_3\subseteq Z_1))\Rightarrow  Z_1\in\nu(x)]$ & \\
iff &$\forall  x\forall Z_1\forall Z_2[(Z_2\in \nu(x)\ \&\ (\forall y\in Z_2)( Z_1\in\nu(y)))\Rightarrow  Z_1\in\nu(x)]$ & Monotonicity\\
\end{tabular}
\end{flushleft}
}}

{\small{
\begin{flushleft}
\begin{tabular}{c ll}
\mc{3}{l}{5.\ \, $\mathbb{F}\models \neg \nabla \neg p\to \nabla \neg  \nabla \neg p \ \rightsquigarrow\ \neg[\nu^c] \langle \not \ni \rangle \neg p \subseteq  [\nu^c] \langle \not \ni \rangle \neg \langle \nu \rangle  [ \ni ] \neg p$} \\
\hline
&$\neg[\nu^c] \langle \not \ni \rangle \neg p \subseteq [\nu^c] \langle \not \ni \rangle \neg \langle \nu \rangle  [ \ni ] \neg p$ &\\ 
iff &$\forall x\forall Z_1[[\nu^c] \langle \not \ni \rangle \neg \langle \nu \rangle  [ \ni ] \neg p\subseteq \{x\}^c\ \&\ \langle\not\ni\rangle \lnot p\subseteq\{Z_1\}^c \Rightarrow \neg[\nu^c] \{Z\}^c \subseteq\{x\}^c ]$ & first approx.\\ 
iff &$\forall x\forall Z_1[[\nu^c] \langle \not \ni \rangle \neg \langle \nu \rangle  [ \ni ] \neg p\subseteq \{x\}^c\ \&\ \lnot[\not\in]  \{Z_1\}^c\subseteq p \Rightarrow \neg[\nu^c] \{Z\}^c \subseteq\{x\}^c ]$ & Residuation\\ 
iff &$\forall x\forall Z_1[[\nu^c] \langle \not \ni \rangle \neg \langle \nu \rangle  [ \ni ] \neg \lnot[\not\in]  \{Z_1\}^c\subseteq \{x\}^c\Rightarrow \neg[\nu^c] \{Z\}^c \subseteq\{x\}^c ]$ & Ackermann\\ 
iff &$\forall Z_1[\neg[\nu^c] \{Z_1\}^c \subseteq[\nu^c] \langle \not \ni \rangle \neg \langle \nu \rangle  [ \ni ] \neg \lnot[\not\in]  \{Z_1\}^c ]$ & \\ 
iff &$\forall Z_1\forall x[xR_{\nu^c}Z_1\Rightarrow\forall Z_2(xR_{\nu^c}Z_2\Rightarrow\exists y(Z_2 R_{\not\ni} y\ \&\ \forall Z_3(yR_{\nu}Z_3\Rightarrow\exists w(Z_3 R_{\ni}w\ \&\ wR_{\notin}Z_1))))]$ & Standard translation\\ 
iff &$\forall Z_1\forall x[Z_1\notin\nu(x)\Rightarrow(\forall Z_2\notin\nu(x))(\exists y\notin Z_2)(\forall Z_3\in\nu(y))(\exists w\in Z_3)(w\notin Z_1)]$ & Relation translation\\ 
iff &$\forall Z_1\forall x[Z_1\notin\nu(x)\Rightarrow(\forall Z_2\notin\nu(x))(\exists y\notin Z_2)(\forall Z_3\in\nu(y))(Z_3\nsubseteq Z_1)]$ & \\ 
iff &$\forall Z_1\forall x[Z_1\notin\nu(x)\Rightarrow\forall Z_2(((\forall y\notin Z_2)(\exists Z_3\in\nu(y))(Z_3\subseteq Z_1))\Rightarrow Z_2\in\nu(x))]$ & Contraposition\\ 
iff &$\forall Z_1\forall x[Z_1\notin\nu(x)\Rightarrow\forall Z_2((\forall y\notin Z_2) (Z_1\in\nu(y))\Rightarrow Z_2\in\nu(x))]$ & Monotonicity \\ 
iff &$\forall Z_1\forall x[Z_1\notin\nu(x)\Rightarrow\{y\mid Z_1\in\nu(y)\}^c\in\nu(x))]$ & Monotonicity \\ 
\end{tabular}
\end{flushleft}
}}

{\small{ 
\begin{flushleft}
\begin{tabular}{c ll}
\mc{3}{l}{B.\ \, $\mathbb{F}\models p \to \nabla \neg  \nabla \neg p\ \rightsquigarrow\  p \subseteq [\nu^c] \langle \not \ni \rangle \neg \langle \nu \rangle  [ \ni ] \neg p$} \\
\hline
&$p \subseteq [\nu^c] \langle \not \ni \rangle \neg \langle \nu \rangle  [ \ni ] \neg p$\\ 
iff & $\forall  x \forall p [  \{x\} \subseteq p \Rightarrow  \{x\} \subseteq [\nu^c] \langle \not \ni \rangle \neg \langle \nu \rangle  [ \ni ] \neg p]$ & first approx. \\
iff & $\forall  x [  \{x\} \subseteq [\nu^c] \langle \not \ni \rangle \neg \langle \nu \rangle  [ \ni ] \neg  \{x\}]$ & Ackermann\\
iff  & $\forall  x [  \{x\} \subseteq [\nu^c] \langle \not \ni \rangle  [ \nu ]  \langle \ni \rangle  \{x\}]$ & \\
iff  & $\forall  x [ \forall Z_1 (x R_{\nu^c} Y \Rightarrow \exists y (Y R_{\not \ni} x \ \& \ \forall Z_2 (y R_\nu Z_2 \Rightarrow Z_2 R_\ni x )))]$ & Standard translation\\
iff  & $\forall  x [ \forall Z_1 (Z_1 \not \in \nu(x) \Rightarrow \exists y (x \not \in Z_1 \ \& \ \forall Z_2 (Z_2 \in \nu(y) \Rightarrow x \in Z_2 )))]$ & Relations translation\\
iff  & $\forall  x [ \forall Z_1(\forall y(\forall Z_2(x\notin Z_2\Rightarrow Z_2\notin\nu(y))\Rightarrow y \in Z_1) \Rightarrow Z_1 \in\nu(x) )]$ & Contrapositive\\
iff  & $\forall  x [ \forall Z_1(\forall y(\{x\}^c\notin\nu(y_1))\Rightarrow y \in Z_1) \Rightarrow Z_1 \in\nu(x) )]$ & Monotonicity\\
iff  & $\forall  x [ \{ y \mid \{x\}^c\notin\nu(y)\}\in\nu(x) )]$ & Monotonicity\\
iff  & $\forall x \forall X[ x \in X \Rightarrow  \{ y\mid X^c \notin \nu(y) \} \in \nu(x)]$ & Monotonicity\\
\end{tabular}
\end{flushleft}
}}

{\small{
\begin{flushleft}
\begin{tabular}{c ll}
\mc{3}{l}{D.\ \, $\mathbb{F}\models \nabla p \to \neg  \nabla \neg p\ \rightsquigarrow\  \langle \nu \rangle  [ \ni ] p \subseteq \neg\langle \nu \rangle  [ \ni ] \neg p$} \\
\hline
&$\langle \nu \rangle   [ \ni ] p \subseteq\neg\langle \nu \rangle  [ \ni ] \neg p$ \\ 
iff & $ \forall Z \forall Z' [ \{Z \} \subseteq [\ni ] p \ \&\  Z' \subseteq \ [\ni] \neg p  \Rightarrow \langle \nu \rangle \{Z \}\subseteq\neg \langle \nu \rangle Z']$ & first approx. \\
iff & $ \forall Z \forall Z' [ \langle \in \rangle \{Z \} \subseteq  p \ \&\  \{Z'\} \subseteq \ [\ni] \neg p  \Rightarrow \langle \nu \rangle \{Z \}\subseteq\neg \langle \nu \rangle \{Z'\}]$ & Residuation \\
iff & $ \forall Z \forall Z' [   \{Z'\} \subseteq \ [\ni] \neg \langle \in \rangle \{Z \}  \Rightarrow \langle \nu \rangle \{Z \}\subseteq\neg \langle \nu \rangle \{Z'\}]$ & $(\star)$ Ackermann \\
iff & $\forall Z [ \langle \nu \rangle \{Z\} \subseteq \neg\langle \nu \rangle  [ \ni ] \neg \langle \in \rangle \{Z\}]$ & \\
iff & $\forall Z [ \langle \nu \rangle \{Z\} \subseteq [ \nu]  \langle \ni \rangle \langle \in \rangle \{Z\}]$ & \\
iff & $\forall Z\forall x [ xR_\nu Z \Rightarrow  \forall Y(xR_\nu Y\Rightarrow \exists w(Y R_\ni w\ \&\ w R_\in Z))]$ & Standard Translation\\
iff & $\forall Z\forall x [ Z\in \nu(x) \Rightarrow  \forall Y(Y\in\nu(x)\Rightarrow \exists w( w\in Y \ \&\ w \in Z))]$ & Relation translation\\
iff & $\forall Z\forall x [ Z\in \nu(x) \Rightarrow  \forall Y(Y\in\nu(x)\Rightarrow Y\nsubseteq Z^c)]$ & \\
iff & $\forall Z\forall x [ Z\in \nu(x) \Rightarrow  \forall Y(Y\subseteq Z^c\Rightarrow Y\notin\nu(x))]$ & Contrapositive\\
iff & $\forall Z\forall x \forall Y[ Z\in \nu(x) \Rightarrow  Z^c\notin\nu(x)]$ & Monotonicity\\
\end{tabular}
\end{flushleft}
}}

{\small{
\begin{flushleft}
	\begin{tabular}{c ll}
\mc{3}{l}{CS.\ \, $\mathbb{F}\models (p \land q) \to (p\succ q)\ \rightsquigarrow\ (p\land q) \subseteq ([\ni] p\land[\not\ni\rangle p){\rhd}  q$} \\
\hline
		&$(p\land q) \subseteq ([\ni] p\dand[\not\ni\rangle p){\rhd}  q $ &\\  
		iff & $\forall  x \forall  Z\forall x' \forall p \forall q[  \{x\} \subseteq p \land q \ \& \  \{Z\}  \subseteq [ \ni ] p \dand [\not \ni \rangle p \ \& \ q  \subseteq \{x'\}^c   \Rightarrow   \{x\} \subseteq  \{Z\}    {\rhd}    \{x'\}^c]$  & first approx. \\ 
		iff & $\forall  x\forall  Z  \forall x\forall p \forall q[  \{x\} \subseteq p \ \& \  \{x\} \subseteq q \ \& \  \{Z\}  \subseteq [ \ni ] p \ \& \  \{Z\}  \subseteq [\not \ni \rangle p \ \& \ q  \subseteq  \{x'\}^c   \Rightarrow   \{x\} \subseteq  \{Z\}    {\rhd}   \{x'\}^c]$ & Splitting rule \\
		iff & $\forall  x \forall  Z\forall x'\forall p \forall q[  \{x\} \subseteq p \ \& \  \{x\} \subseteq q \ \& \  \{Z\}  \subseteq [ \ni ] p \ \& \ p \subseteq [ \not  \in \rangle  \{Z\}   \ \& \ q  \subseteq  \{x'\}^c   \Rightarrow   \{x\} \subseteq  \{Z\}    {\rhd}  \{x'\}^c]$ & Residuation \\
		iff & $\forall  x \forall  Z  \forall x'  \forall q[  \{x\} \subseteq [ \not  \in \rangle  \{Z\}  \ \& \  \{x\} \subseteq q \ \& \  \{Z\}  \subseteq [ \ni ] [ \not  \in \rangle  \{Z\}    \ \& \ q  \subseteq  \{x'\}^c   \Rightarrow   \{x\} \subseteq  \{Z\}    {\rhd}   \{x'\}^c]$ & Ackermann \\
		iff & $\forall  x\forall  Z \forall x' [  \{x\} \subseteq [ \not  \in \rangle  \{Z\}   \ \& \  \{Z\}  \subseteq [ \ni ] [ \not  \in \rangle  \{Z\}    \ \& \  \{x\}  \subseteq  \{x'\}^c   \Rightarrow   \{x\} \subseteq  \{Z\}    {\rhd}  \{x'\}^c]$ & $(\star)$ Ackermann \\
		iff & $\forall  x \forall  Z  [  \{x\} \subseteq [ \not  \in \rangle  \{Z\}   \ \& \  \{Z\}  \subseteq [ \ni ] [ \not  \in \rangle  \{Z\}    \Rightarrow   \{x\} \subseteq  \{Z\}    {\rhd}    \{x\}]$ &  \\
		iff & $\forall  x \forall  Z  [  \lnot x R_{\not\in} Z   \ \& \  \forall y(Z R_\ni y\Rightarrow \lnot y R_{\not \in} Z)   \Rightarrow   \forall y( T_f(x,Z,y) \Rightarrow y =x)]$ & Standard translation \\
		iff & $\forall  x \forall  Z  [   x\in Z   \ \& \  \forall y(y\in Z\Rightarrow  Z\in y )   \Rightarrow   \forall y( y\in f(x,Z) \Rightarrow y=x)]$ & Relation interpretation \\
		iff & $\forall  x \forall  Z  [   x\in Z \   \Rightarrow \   \forall y( y\in f(x,Z) \Rightarrow y=x)]$ & \\
		iff  & $\forall x\forall Z[x \in Z \Rightarrow f(x,Z) \subseteq \{x\}]$ & \\
	\end{tabular}
\end{flushleft}
}}

{\small{
\begin{flushleft}
	\begin{tabular}{c ll}
\mc{3}{l}{CEM.\ \, $\mathbb{F}\models (p \succ q) \lor (p\succ \neg q) \ \rightsquigarrow\ (([\ni] p\dand[\not\ni\rangle p){\rhd}  q)\lor (([\ni] p\dand[\not\ni\rangle p){\rhd}  \neg q)$} \\
\hline
		&$\top \subseteq (([\ni] p\dand[\not\ni\rangle p){\rhd}  q)\lor (([\ni] p\dand[\not\ni\rangle p){\rhd}  \neg q)$ &\\  
		iff & $\forall p \forall q\forall X \forall Y \forall x \forall y   (\{X\} \subseteq [\ni] p\dand[\not\ni\rangle p \ \& \ $ & \\
		& $\{Y\} \subseteq [\ni] p\dand[\not\ni\rangle p\ \& \ q \subseteq \{x\}^c \ \& \ \{y\} \subseteq q \Rightarrow \top \subseteq (\{X\}{\rhd}\{x\}^c)\lor (\{Y\}{\rhd}\neg \{y\}) $  & first approx. \\ 
		iff & $\forall p \forall q\forall X \forall Y \forall x \forall y   (\{X\} \subseteq [\ni] p \ \& \ \{X\} \subseteq [\not\ni\rangle p \ \& \ $ & \\
		& $\{Y\} \subseteq [\ni] p\ \& \ \{Y\} \subseteq [\not\ni\rangle p\ \& \ q \subseteq \{x\}^c \ \& \ \{y\} \subseteq q \Rightarrow \top \subseteq (\{X\}{\rhd}\{x\}^c)\lor (\{Y\}{\rhd}\neg \{y\}) $  & $(\star)$ Splitting \\ 
		iff & $\forall p \forall q\forall X \forall Y \forall x \forall y   (\{X\} \subseteq [\ni] p \ \& \ p \subseteq [\not\in\rangle \{X\} \ \& \ $ & \\
		& $\{Y\} \subseteq [\ni] p\ \& \ p \subseteq [\not\in\rangle \{Y\} \ \& \ q \subseteq \{x\}^c \ \& \ \{y\} \subseteq q \Rightarrow \top \subseteq (\{X\}{\rhd}\{x\}^c)\lor (\{Y\}{\rhd}\neg \{y\}) $  & Residuation \\ 
		iff & $\forall X \forall Y \forall x \forall y   (\{X\} \lor \{Y\} \subseteq [\ni] ([\not\in\rangle \{X\} \land [\not\in\rangle \{Y\})   \ \& \ $ & \\
		& $\{y\} \subseteq \{x\}^c \Rightarrow \top \subseteq (\{X\}{\rhd}\{x\}^c)\lor (\{Y\}{\rhd}\neg \{y\}) $  & Ackermann \\
		iff & $\forall X \forall Y \forall x    (\{X\} \lor \{Y\} \subseteq [\ni] ([\not\in\rangle \{X\} \land [\not\in\rangle \{Y\})   \Rightarrow \forall y( \{y\} \subseteq \{x\}^c \Rightarrow \top \subseteq (\{X\}{\rhd}\{x\}^c)\lor (\{Y\}{\rhd}\neg \{y\})) $  & Currying \\ 
		iff & $\forall X \forall Y \forall x    (\{X\} \lor \{Y\} \subseteq [\ni] ([\not\in\rangle \{X\} \land [\not\in\rangle \{Y\})   \Rightarrow \top \subseteq (\{X\}{\rhd}\{x\}^c)\lor (\{Y\}{\rhd}\neg \{x\}^c)) $  &  \\
		iff & $\forall X \forall Y \forall x[(\forall y(XR_\ni y\ \text{ or }\ YR_\ni y)\Rightarrow\lnot yR_{\not\in} X\ \&\ \lnot yR_{\not\in} Y)\ \Rightarrow \forall y(\lnot T_f(y,X,x)\ \text{ or }\ (\forall z( T_f(y,Y,z)\Rightarrow z=x)))]$  & Standard translation \\
		iff & $\forall X \forall Y \forall x[(\forall y(y\in X \text{ or }\ y\in Y)\Rightarrow y\in X\ \&\ y\in Y)\ \Rightarrow \forall y(x\notin f(y,X)\ \text{ or }\ (\forall z( z\in f(y,Y)\Rightarrow z=x)))]$  & Relation interpretation \\
		iff & $\forall X \forall Y \forall x[(X\cup Y\subseteq  X\cap Y)\ \Rightarrow \forall y(x\notin f(y,X)\ \text{ or }\ (\forall z( z\in f(y,Y)\Rightarrow z=x)))]$  & \\
		iff & $\forall X \forall Y \forall x[X=Y \Rightarrow \forall y(x\notin f(y,X)\ \text{ or }\ (\forall z( z\in f(y,Y)\Rightarrow z=x)))]$  & \\	
		iff & $\forall X \forall x\forall y[(x\notin f(y,X)\ \text{ or }\ (\forall z( z\in f(y,X)\Rightarrow z=x)))]$  & \\
		iff & $\forall X \forall x\forall y[(x\in f(y,X)\ \Rightarrow\ f(y,X)=\{x\})]$  & \\
		iff & $\forall X \forall y[|f(y,X)|\leq 1]$  &		
\end{tabular}
\end{flushleft}
}}

{\small{
\begin{flushleft}
	\begin{tabular}{c ll}
\mc{3}{l}{ID.\ \, $\mathbb{F}\models p\succ p \ \rightsquigarrow\ ([\ni] p\dand[\not\ni \rangle p){\rhd} p$} \\
\hline
		&$\top\subseteq ([\ni] p\dand[\not\ni\rangle p){\rhd} p$ &\\  
		iff & $\forall  Z \forall Z'\forall x'\forall p [( \{Z\} \subseteq [\ni] p \ \&\ \{Z'\}\subseteq [\not\ni\rangle p \ \&\ p\subseteq  \{x'\}^c)\Rightarrow  \top \subseteq( \{Z\} \dand\{Z'\}){\rhd}   \{x'\}^c]$ & first approx. \\ 
		iff & $\forall  Z \forall Z'\forall x'\forall p [(\langle\in\rangle \{Z\} \subseteq p \ \&\ \{Z'\}\subseteq [\not\ni \rangle p \ \&\ p\subseteq \{x'\}^c)\Rightarrow  \top\subseteq( \{Z\} \dand\{Z'\}){\rhd}   \{x'\}^c]$ & Adjunction \\
		iff & $\forall  Z \forall Z'\forall  x' [(\{Z'\}\subseteq [\not\ni \rangle \langle\in\rangle \{Z\}  \ \&\ \langle\in\rangle \{Z\} \subseteq  \{x'\}^c)\Rightarrow  \top\subseteq( \{Z\} \dand\{Z'\}){\rhd}  \{x'\}^c$ & Ackermann \\
		iff & $\forall  Z \forall Z'[\{Z'\}\subseteq [\not\ni \rangle \langle\in\rangle \{Z\}  \ \Rightarrow \forall  x' [ \langle\in\rangle \{Z\} \subseteq  \{x'\}^c \Rightarrow  \top \subseteq( \{Z\} \dand\{Z'\}){\rhd}  \{x'\}^c]]$ & Currying \\
		iff & $\forall  Z\forall Z'[\{Z'\}\subseteq [\not\ni \rangle \langle\in\rangle \{Z\}  \ \Rightarrow  \top \subseteq( \{Z\} \dand\{Z'\}){\rhd}  \langle\in\rangle \{Z\} ]$ & $(\star)$ Ackermann \\
		iff & $\forall x\forall  Z\forall Z'[ \forall w(Z'R_{\not\ni}w\Rightarrow \lnot wR_\in Z) \Rightarrow \forall y (T_f (x, Z, y ) \ \& \ Z=Z'   \Rightarrow y \in Z) ]$ & Standard Translation \\
		iff & $\forall x\forall  Z\forall Z'\forall y[ \forall w(Z'R_{\not\ni}w\Rightarrow \lnot wR_\in Z) \ \&\   (T_f (x, Z, y ) \ \& \ Z=Z'   \Rightarrow y \in Z) ]$ & \\
		iff & $\forall x\forall  Z\forall Z'\forall y[ \forall w(w\notin Z'\Rightarrow  w\notin Z) \ \&\   (y\in f (x, Z) \ \& \ Z=Z'   \Rightarrow y \in Z) ]$ & Relation interpretation\\
		iff & $\forall x\forall  Z\forall Z'\forall y[ Z\subseteq Z' \ \&\   (y\in f (x, Z) \ \& \ Z=Z'   \Rightarrow y \in Z) ]$ & \\
		iff & $\forall x\forall  Z\forall y[(y\in f (x, Z)\Rightarrow y \in Z) ]$ & \\
		iff  & $\forall x\forall Z[f(x,Z)\subseteq Z]$ & \\		
	\end{tabular}
\end{flushleft}
}}

\bibliography{ref}
\bibliographystyle{plain}

\end{document}